\documentclass[fleqn]{article}

\usepackage{amsmath,amssymb,amsthm,tikz,enumerate,stmaryrd}
\usepackage[f]{esvect}
\usepackage{bbold}

\title{Pour-El's Landscape}
\author{Taishi Kurahashi\footnote{Email: kurahashi@people.kobe-u.ac.jp}
\footnote{Graduate School of System Informatics, Kobe University, Japan.} \ 
and Albert Visser\footnote{Email: a.visser@uu.nl}\ \footnote{Philosophy, Faculty of Humanities, Utrecht University, The Netherlands} 
}
\date{}

\theoremstyle{plain}
\newtheorem{thm}{Theorem}[section]
\newtheorem{lem}[thm]{Lemma}
\newtheorem{prop}[thm]{Proposition}
\newtheorem{cor}[thm]{Corollary}

\theoremstyle{definition}
\newtheorem{defn}[thm]{Definition}
\newtheorem{ex}[thm]{Example}
\newtheorem{rem}[thm]{Remark}
\newtheorem{prob}[thm]{Problem}

\renewcommand{\phi}{\varphi}

\renewcommand{\Theta}{\varTheta}
\renewcommand{\Phi}{\varPhi}
\renewcommand{\Psi}{\varPsi}
\renewcommand{\Xi}{\varXi}
\renewcommand{\Omega}{\varOmega}
\renewcommand{\Gamma}{\varGamma}
\newcommand{\qedright}{\belowdisplayskip=-12pt}
\newcommand{\PA}{\mathsf{PA}}

\newcommand{\Prf}{\mathsf{Prf}}

\newcommand{\Con}{\mathsf{Con}}
\newcommand{\gn}[1]{\ulcorner{#1}\urcorner}

\newcommand{\num}[1]{\underline{#1}}

\newcommand{\mc}[1]{\mathcal {#1}}

\newcommand{\RQ}{\mathsf{Q}}

\newcommand{\RR}{\mathsf{R}}

\newcommand{\verz}[1]{\{ #1 \}}

\renewcommand{\iff}{\leftrightarrow}

\newcommand{\mf}[1]{{\mathfrak {#1}}}
 
 \newcommand{\To}{\Rightarrow}
  \newcommand{\tupel}[1]{{\langle #1 \rangle}}
   
   \newcommand{\adj}[2]{#1 \cup \{#2\}}
   \newcommand{\jump}{\mathrel{\mbox{\textcolor{gray}{$\blacktriangleright$}}}}
\newcommand{\jumpb}{\mathrel{\mbox{\textcolor{gray}{$\blacktriangleleft$}}}}
\newcommand{\domi}{\mathrel{\Yleft}}

\newcommand{\sline}{\raise-0.3ex\hbox{$\hbox{--}\kern-0.84ex\raise0.45ex\hbox{$\hbox{\scalebox{0.3}{\bf /}}\kern-0.37ex\hbox{\scalebox{0.3}{\bf /}}$}$}}
\newcommand{\slinei}{\raise-0.3ex\hbox{$\hbox{--}\kern-0.84ex\raise0.45ex\hbox{$\hbox{\scalebox{0.3}
{\bf \textbackslash}}\kern-0.37ex\hbox{\scalebox{0.3}{\bf \textbackslash}}$}$}}
\newcommand{\jumpneq}{\mathrel{\jump_{\hspace*{-0.235cm}{}_{\kern0.08ex \sline}}\hspace*{0.09cm}}}
\newcommand{\jumpbneq}{\mathrel{\jumpb_{\hspace*{-0.235cm}{}_{\kern0.08ex \slinei}}\hspace*{0.09cm}}}
\newcommand{\dashvneq}{\mathrel{\dashv_{\hspace*{-0.235cm}{}_{\kern0.2ex \slinei}}\hspace*{0.09cm}}}
\newcommand{\vdashneq}{\mathrel{\vdash_{\hspace*{-0.235cm}{}_{\kern0.3ex \sline}}\hspace*{0.09cm}}}
\newcommand{\lhdnneq}{\mathrel{\lhd_{\hspace*{-0.27cm}{}_{\kern0.2ex \slinei}}\hspace*{0.09cm}}}
\newcommand{\rhdnneq}{\mathrel{\rhd_{\hspace*{-0.27cm}{}_{\kern0.3ex \sline}}\hspace*{0.09cm}}}

\definecolor{uured}{cmyk}{0.2,1,0.9,0.1}
\definecolor{uublue}  {cmyk}{0.9,0.55,0,0}
\definecolor{uugreen}{cmyk}{1,0,0.75,0}
\definecolor{bazaar}{rgb}{0.6, 0.47, 0.48}

\DeclareMathOperator{\possible}{\text{\tikz[scale=.6ex/1cm,baseline=-.6ex,rotate=45,line width=.1ex]{
                            \draw (-1,-1) rectangle (1,1);}}}
\DeclareMathOperator{\necessary}{\text{\tikz[scale=.6ex/1cm,baseline=-.6ex,line width=.1ex]{
                            \draw (-1,-1) rectangle (1,1);}}}

\newcommand{\apr}{{\vartriangle}}
\newcommand{\aco}{{\triangledown}}
\newcommand{\opr}{\necessary}
\newcommand{\oco}{\possible}

\newcommand{\ce}{c.e.}
\newcommand{\compu}{computable}
\newcommand{\compy}{computably}
\newcommand{\compab}{computability}

\begin{document}

\maketitle

\begin{abstract}
We study the effective versions of several notions related to incompleteness, undecidability and inseparability along the lines of Pour-El's insights. 
Firstly, we strengthen Pour-El's theorem on the equivalence between effective essential incompleteness and effective inseparability. 
Secondly, we compare the notions obtained by restricting that of effective essential incompleteness to intensional finite extensions and extensional finite extensions.
Finally, we study the combination of effectiveness and hereditariness, and prove an adapted version of Pour-El's result for this combination. 
\end{abstract}

\section{Introduction}
Reflection on the meaning of incompleteness and undecidability results gave rise to notions like \emph{essential undecidability} of theories
and \emph{\compu\ inseparability} of theories: a consistent c.e.~theory $U$ is essentially undecidable iff all its consistent extensions $V$ (in the same language)
are undecidable, and  a consistent \ce~theory $U$ is \compy\ inseparable iff its theorems and its refutable sentences are \compy\ inseparable.
We note that these notions have different flavours: \emph{essential undecidability} looks at the relation of the given theory with other theories and
\emph{\compu\ inseparability} looks at the relation of the theory with \ce~sets. 

Such notions were studied in the period 1950--1970, see e.g.~\cite{Smullyan61, Smullyan93}. 
Their various relations and non-relations were established. See the schema at the end of Section~\ref{noneffsmurf}. 

Marian Boykan Pour-El, in her ground-breaking paper~\cite{pour:effe68}, made an interesting discovery. Where there are examples of essentially incomplete
theories that are not \compy\ inseparable, the effective versions of these notions coincide. The present paper is a study of results along
the lines of Pour-El's insight. We study effective versions of notions connected to incompleteness and undecidability and establish their interrelationships.
See the schema at the end of Section~\ref{heredity} for an overview of our results.

In the usual statements of incompleteness results, there is often a restriction to some specific formula class.
For example: for all \ce\ extensions $U$ of the Tarski-Mostowski-Robinson theory {\sf R}, we can effectively find a $\Sigma^0_1$-sentence
$\sigma$ that is independent of $U$ from an index of $U$. The Pour-El style results in this
paper will reflect the possibility of such restrictions: we will add a parameter for the formula class from 
which the witnesses of e.g. incompleteness or creativity may be chosen. 

Effective versions of incompleteness and undecidability results have unavoidably an intensional component. For example, a theory $U$
is \emph{effectively essentially incomplete} iff there is a partial \compu\ function $\Phi$ such that, for all indices $i$, if
the \ce~set ${\sf W}_i$ is a consistent extension of $U$, then $\Phi(i)$ converges to a sentence that is independent of ${\sf W}_i$.
We see that the function that provides the independent sentences operates on presentations rather than directly on the extensions of the
given theory itself. Thus, our paper is as much a study of the consequences of intensionality as it is of effectiveness.
In the case of finite extensions, we can operate more extensionally, since we can consider these as given not by an index
but rather by a sentence. This fact enables us, in the special case of finite extensions, to compare the intensional and the
extensional. In Section~\ref{sec_ef}, we study the extensional finite case.

\subsection{Overview of the Paper}
In Section~\ref{notions}, we introduce the basic notions. We present an analysis of what the effectivisation of a notion is in
Section~\ref{whatsmurf}. Section~\ref{pour-el-sec} gives our presentation and extension of Pour-El's work. Then, in Section~\ref{sec_ef},
we study the extensional case for finite extensions. Section~\ref{heredity} is concerned with the combination of effectiveness and hereditariness. 
Among other things, we prove that for any consistent \ce~theory, effective essential hereditary creativity and strong effective inseparability are equivalent. 
The section is a sequel to the paper~\cite{Vis22}, where the non-effective case of hereditariness is studied.

Apart from reading the paper from A to Z, there are several other paths the reader may beneficially follow. 
Section~\ref{notions} is more or less obligatory in order to understand the rest, 
but e.g. Section~\ref{nijveresmurf} can be read lightly to return to
when the relevant notions are needed. Of Section~\ref{pour-el-sec}, the presentation of Pour-El's original result in
Section~\ref{mainsmurf} should not be skipped, but
the reader could cherry-pick from the rest. Then, the reader can choose between Section~\ref{sec_ef} and Section~\ref{heredity}.

\section{Notions and Basic Facts}\label{notions}

In this section, we introduce the various notions that we employ in this paper and present some basic facts.

\subsection{Theories}
A \emph{theory} is given as a set of axioms in a given signature. We take the signature to be part of the
data of the theory. Note that we do not take a formula representing the axiom set or an index of
the axiom set as part of the data. The same theory can have different enumerations. Moreover, these
enumerations are enumerations of the axiom set and not of the theorems.

We write $U_{\mf p}$ for the set of theorems of $U$ and $U_{\mf r}$ for the $U$-refutable sentences or
anti-theorems of $U$, i.e. $U_{\mf r} = \verz{\phi \mid U \vdash \neg\, \phi}$.

We will also consider \emph{mono-consequence}: $U \vdash_{\mf m} \phi$ iff there is a $\psi\in U$ such that
$\psi \vdash \phi$. We have the corresponding notion of mono-consistency, which was developed by Lindstr\"om (cf.~\cite{Lin03}).
We write $U_{\mf m}$ for the set of mono-theorems of $U$ and $U_{\mf n}$ for the set of mono-refutable sentences.

Strictly speaking, there is no disjoint notion of \emph{mono-theory}. A mono-theory is just a theory. However, sometimes we will
still use the word to indicate that we intend to use the given set of sentences of the given signature with the notion of mono-consequence.

The notion of mono-theory plays a role in Theorem~\ref{effehr3}. Below we show the notion of mono-theory is connected to the idea of \compu\ sequence of \ce~theories.
There are heuristic differences between the notions. Sometimes it is better to think in terms of the `flatter' notion of
mono-theory, sometimes it is pleasant to visualise a sequence of theories.

We define $\widehat U$ as the set of all non-empty finite conjunctions of elements of $U$. We have the following convenient insight. 

\begin{thm}
    We have: $U\vdash \phi$ iff $\widehat U \vdash_{\mf m} \phi$.
    As a consequence, $U_{\mf p} = \widehat U_{\mf m}$ and $U_{\mf r} = \widehat U_{\mf n}$.
\end{thm}

We note that we have $V \dashv U$ iff $V \dashv_{\mf m} \widehat U$. So, $\widehat{(\cdot)}$ is the right adjoint of the projection
functor of theories (of a given signature) with $\dashv$ into theories (of the given signature) with $\dashv_{\mf m}$.
We also note that, of course, $(\cdot)_{\mf p}$  gives us an isomorphic functor. The advantage of  $\widehat{(\cdot)}$ is the
fact that it does not raise the complexity of the given set for most measures of complexity.

A \emph{\compu\ sequence $(T_i)_{i\in \omega}$ of \ce~theories} is given by
a \ce~relation $\mc T(i,\phi)$. Here, of course, $T_i := \verz{\phi \mid \mc T(i,\phi)}$.
We demand that all $T_i$ are in the same signature. We take this signature as
part of the data of the sequence. A sequence of theories is consistent if each of its
theories is.

Consider a \compu\ sequence of \ce~theories, given by $\mc T(i,\phi)$.
We define $\mc T^{\mf m} := \bigcup_{i\in \omega} \widehat T_i$. Clearly,
$\mc T^{\mf m}$ is a \ce~mono-theory.
It is easy to see that $\mc T^{\mf m} \vdash_{\mf m}\phi$ iff $T_i \vdash \phi$, for some $i$.

Conversely, given a \ce~mono-theory $U$ we can associate 
a \compu\ sequence $(T_i)_{i\in \omega}$ of \ce~theories as follows.
Enumerate $U$ in stages and, if, at stage $i$, a sentence $\phi$ is enumerated,
put $\phi$ in $T_i$. Clearly, $U \vdash_{\mf m}\phi$ iff $T_i \vdash \phi$, for some $i$.

Since metamathematical results on sequences of theories are mostly concerned with the relation
$\exists i\in \omega\;T_i \vdash \phi$, we can usually replace sequences of theories given by $\mc T$ by 
mono-theories $U$ and study $U \vdash_{\mf m}\phi$.

\subsection{Theory-Extension}\label{nijveresmurf}
We may define various notions of \emph{theory-extension}. The basic notion is simply $U\subseteq V$: the $V$-axioms extend the $U$-axioms.
Here the $V$-language may extend the $U$-language. We have three `dimensions' of variation. 
\begin{enumerate}[i.]
    \item We can put restrictions on the $V$-language.
We consider two possibilities. We use  a superscript $\mf s$, for `same' or for `signature', to indicate that the $U$- and the
$V$-language coincide. We use a superscript $\mf c$ to indicate that the $V$-language extends the $U$-language by at most finitely
many constants.
\item 
We do not compare the axiom sets but appropriate closures of the axioms sets. When we compare the theorems, we indicate this by a subscript $\mf p$.
We can also compare the mono-theorems. We indicate this by a subscript $\mf m$.
\item 
We may put constraints on the cardinality of the extension. We use the subscript $\mf f$ for finite extensions.
\end{enumerate}

So we will use, e.g. $U\subseteq^{\mf s}_{\mf {f}}V$, for: $V$ is a finite extension of $U$ in the same language.
If we use, e.g., $U\subseteq^{\mf s}_{\mf {pf}}V$, this is of course intended to mean that the theorems of $V$ are theorems
of a finite extension of $U$. We will use $U\dashv V$ for $U \subseteq^{\mf s}_{\mf p} V$ and $U\dashv_{\mf m} V$ for 
$U \subseteq^{\mf s}_{\mf m} V$.

We note that if $U_{\mf p} = U'_{\mf p}$ and 
$V_{\mf p} = V'_{\mf p}$, then $U\subseteq^{\mf s}_{\mf {pf}}V$ iff
$U'\subseteq^{\mf s}_{\mf {pf}}V'$. 

We will be looking at mono-extensions of non-mono theories. For this case the following notion of extension is a relevant one.
\begin{itemize}
\item
 $U \Cup V := \verz{\phi \wedge \psi \mid \phi\in U \text{ and } \psi\in V}$.
 \item
$U \Subset V$ iff $U\Cup V \subseteq^{\mf s}_{\mf m} V$.
\end{itemize}

Let $\mc X$ and $\mc Y$ be disjoint \ce~sets. 
Two sets $\mc Z$ and $\mc W$ \textit{weakly biseparate} $\mc X$ and $\mc Y$ iff $\mc X \subseteq \mc Z$, $\mc Y \subseteq \mc W$, $\mc Z \cap \mc Y = \emptyset$, and $\mc W \cap \mc X = \emptyset$. 
We say that $\mc Z$ and $\mc W$ \textit{biseparate} $\mc X$ and $\mc Y$ iff they weakly biseparate $\mc X$ and $\mc Y$ and $\mc Z \cap \mc W = \emptyset$. 
We will not use the following theorem later, but we state it for the sake of understanding.

\begin{thm}
\begin{enumerate}[a.]
\item
If $U \Subset V$, then $\widehat U \subseteq^{\mf s} U_{\mf p} \subseteq^{\mf s}_{\mf m} U_{\mf p} \Cup V \subseteq^{\mf s}_{\mf m} V$. 
\item
If $U \Subset V$ and $V$ is mono-consistent, then $V_{\mf m}$ and $V_{\mf n}$ weakly biseparate $U_{\sf p}$ and $U_{\mf r}$.
\item 
$U\Subset U$ iff $U_{\mf m}  = U_{\mf p}$. 
\end{enumerate}
\end{thm}
\begin{proof}
(a). Suppose $U \Subset V$. 
The inclusions $\widehat U \subseteq^{\mf s} U_{\mf p} \subseteq^{\mf s}_{\mf m} U_{\mf p} \Cup V$ are obvious. 
To prove $U_{\mf p} \Cup V \subseteq^{\mf s}_{\mf m} V$, it suffices to show that for any $k$, $\phi_0, \ldots, \phi_k \in U$, and $\psi \in V$, there exists a $\rho \in V$ such that $\rho \vdash \phi_0 \land \cdots \land \phi_k \land \psi$. 
We prove the statement by induction on $k$, and the case of $k = 0$ is immediate from $U \Subset V$. 
Assume that the statement holds for $k$ and let $\phi_0, \ldots, \phi_k, \phi_{k+1} \in U$ and $\psi \in V$. 
By the induction hypothesis, there exists a $\rho \in V$ such that $\rho \vdash \phi_0 \land \cdots \land \phi_k \land \psi$. 
Since $U \Subset V$, there exists a $\rho' \in V$ such that $\rho' \vdash \phi_{k+1} \land \rho$. 
We obtain $\rho' \vdash \phi_0 \land \cdots \land \phi_k \land \phi_{k+1} \land \psi$. 

(b). Suppose $U \Subset V$ and $V$ is mono-consistent. 
Since $U_{\mf p} \subseteq_{\mf m} V$ by (a), we have $U_{\mf p} \subseteq V_{\mf m}$ and $U_{\mf r} \subseteq V_{\mf n}$. 
If $\xi \in U_{\mf p} \cap V_{\mf n}$ for some sentence $\xi$, then there would be a $\psi \in V$ such that $\psi \vdash \neg \xi$. 
Since $U_{\mf p} \Cup V \subseteq^{\mf s}_{\mf m} V$, there would be a $\rho \in V$ such that $\rho \vdash \xi \land \psi$. 
Then, $\rho$ is inconsistent. 
This contradicts the mono-consistency of $V$. 
Therefore, $U_{\mf p} \cap V_{\mf n} = \emptyset$. 
In a similar way, we can prove $U_{\mf r} \cap V_{\mf m} = \emptyset$. 

(c). By (a), $U \Subset U$ is equivalent to $U_{\mf p} \Cup U \subseteq^{\mf s}_{\mf m} U$, and to $U_{\mf p} \subseteq^{\mf s}_{\mf m} U$. 
Also, $U_{\mf p} \subseteq^{\mf s}_{\mf m} U$ is equivalent to $U_{\mf m} = U_{\mf p}$. 
\end{proof}

\subsection{Interpretability}

An interpretation $K$ of a theory $U$ in a theory $V$ is based on \emph{a translation} of the $U$-language into the $V$-language.
This translation commutes with the propositional connectives. In some broad sense, it also commutes with the
quantifiers but here there are a number of extra features. 
\begin{itemize}
\item
Translations may be more-dimensional: we allow a variable to be translated to an appropriate sequence of variables.
\item
We may have domain relativisation: we allow the range of the translated quantifiers to be some domain definable
in the $V$-language.
\item
We may even allow the new domain to be built up from pieces of, possibly, different dimensions.
\end{itemize} 
A further feature is that identity need not be translated to identity but can be translated to a congruence relation.
 Finally, we may also allow parameters in an interpretation.
To handle these the translation may specify a parameter-domain. 
For details on the various kinds of translation, we refer the reader to~\cite{Vis17}.

We can define the obvious identity translation of a language in itself and composition of translations.

\emph{An interpretation} is a triple $\tupel{U,\tau,V}$, where $\tau$ is a translation of the $U$-language in the
$V$-language such that, for all $\phi$, if $U \vdash \phi$, then $V \vdash \phi^\tau$.\footnote{In case we have parameters
with parameter-domain $\alpha$ this becomes: $V \vdash \exists \vv x\, \alpha(\vv x)$ and,
for all $\phi$, if $U \vdash \phi$, then $V \vdash \forall \vv x\, (\alpha(\vv x) \to \phi^{\tau,\vv x})$}

We write:
\begin{itemize}
\item
$K:U \lhd V$ for: $K$ is an interpretation of $U$ in $V$.
\item
$U \lhd V$ for: there is a $K$ such that $K:U \lhd V$. We also write
$V\rhd U$ for: $U \lhd V$.
\item
$U \lhd_{\sf loc} V$ for: for every finitely axiomatisable sub-theory $U_0$ of $U$, we have
$U_0 \lhd V$.
\item
 $U \lhd_{\sf mod} V$ for: for every $V$-model $\mc M$, there is a translation $\tau$ from the $U$-language in the $V$-language, such that
 $\tau$ defines an internal $U$-model $\mc N = \widetilde \tau(\mc M)$ of $U$ in $\mc M$.
\end{itemize}

In Appendix~\ref{locosmurf}, we will have a brief look at effective versions of local interpretability, essentially concluding that
all such versions collapse either to ordinary local interpretability or, somewhat surprisingly, to global interpretability.

In~\cite{Vis22}, the relation of \emph{essential tolerance} was studied since it has backwards preservation of essential 
hereditary undecidability. In Appendix~\ref{effesstol}, we briefly consider \emph{effective essential tolerance}.
This relation has backwards preservation of effective essential 
hereditary undecidability.

\subsection{The Non-Effective Notions}\label{noneffsmurf}

In this subsection, we introduce the non-effective notions. 
We will then discuss what the appropriate corresponding effective versions should be in the next subsection.

Our first building blocks are decidability, completeness and separability.
\begin{itemize}
    \item A theory $U$ is \emph{decidable} if there is an algorithm that decides provability in $U$.
    In other words, $U$ is decidable iff $U_{\mf p}$ is \compu.
    A theory is \emph{undecidable} if it is not decidable.
    \item 
    A theory $U$ is \emph{complete} if, for every $U$-sentence $\phi$, we have $U\vdash \phi$ or $U \vdash \neg\,\phi$.
    In other words, $U$ is complete iff $U_{\mf p} \cup U_{\mf r}= {\sf Sent}_U$.
    A theory is \emph{incomplete} if it is not complete.
\end{itemize}

Suppose $\mathcal P$ is a property of theories.
We say that $U$ is \emph{essentially $\mathcal P$} if all consistent \ce~extensions
(in the same language) of $U$ are $\mathcal P$. We say that $U$ is \emph{hereditarily $\mathcal P$} if all consistent \ce~sub-theories
of $U$ (in the same language) are $\mathcal P$. We say that $U$ is \emph{potentially $\mathcal P$} if some consistent \ce~extension
(in the same language) of $U$ is $\mathcal P$. 

We defined essential and potential and hereditary with respect to $\subseteq^{\mf s}$.
In this paper we also will consider these notions with respect to $\subseteq^{\mf s}_{\mf f}$.

If $\mathcal R$ is a relation between theories the use of \emph{essential} and \emph{hereditary} and \emph{potential} always
concerns the first component aka the subject. Thus, e.g., we say that \emph{$U$ essentially tolerates $V$} meaning that 
$U$ essentially has the property of tolerating $V$. Tolerance itself is defined as potential interpretation.
So $U$ essentially tolerates $V$ if $U$ essentially potentially interprets $V$.
We will have a closer look at essential tolerance in Appendix~\ref{effesstol}.

An important recursion theoretic notion is \compu\ (in)separability. Two \ce~sets $\mc X$ and $\mc Y$ are \emph{\compy\ separable}
iff, there is a \compu\ $\mc Z$ such that $\mc X\subseteq \mc Z$ and $\mc Y \subseteq \mc Z^{\sf c}$.
Two sets are \emph{\compy\ inseparable} iff they are not \compy\ separable.
We want to apply \compu\  (in)separability to theories and pairs of theories by designating certain sets
of sentences associated with the theories as candidates for \compu\ (in)separability.

Let us say that a pair of theories $(U,V)$ is \emph{acceptable} iff $U$ and $V$ have the same signature and are jointly consistent.
Let $(U,V)$ be acceptable. We define:
\begin{itemize}
    \item $(U,V)$ is \emph{\compy\ \textup(in\textup)separable} iff $U_{\mf p}$ and $V_{\mf r}$ are \compy\ (in)sep\-a\-ra\-ble.
    \item $U$ is \emph{\compy\ \textup(in\textup)separable} iff $(U,U)$ is \compy\ (in)separable.
    \item $U$ is \emph{strongly \compy\ \textup(in\textup)separable} iff $(U, 0_U)$ is \compy\ (in)sep\-a\-ra\-ble.
\end{itemize}
Here, $0_U$ denotes the pure predicate calculus in the language of $U$. 

We define: $(\mc X,\mc Y) \leq_1 (\mc Z,\mc W)$ iff there is an injective \compu\ function $f$, such that
$n\in \mc X$ iff $f(n) \in \mc Z$, and $n\in \mc Y$ iff $f(n) \in \mc W$. Our definition generalises
\cite[Definition 2.4.9, p40]{soar:turi16}, which coincides with our definition when we restrict ourselves
to disjoints pairs of sets. We have $(\mc X,\mc X) \leq_1 (\mc Y,\mc Y)$ iff $\mc X\leq_1\mc Y$.
Clearly, if $(\mc X,\mc Y)$ and $(\mc Z,\mc W)$ are disjoint pairs and if
$(\mc X,\mc Y)$ is \compy\ inseparable and $(\mc X,\mc Y) \leq_1(\mc Z,\mc W)$, then
$(\mc Z,\,\mc W)$ is \compy\ inseparable.

We have the following simple insights:

\begin{thm}
$(U,V)$ is \compy\ inseparable iff $(V,U)$ is \compy\ inseparable.
\end{thm}

\begin{proof}
We note that negation witnesses that $(U_{\mf p},V_{\mf r}) \leq_1  (U_{\mf r},V_{\mf p})$.
So if $(U_{\mf p},V_{\mf r})$ is \compy\ inseparable, then so is  $(V_{\mf p},U_{\mf r})$.
\end{proof}

\begin{thm}[Subtraction Theorem]\label{subtraction}
If $(U+\phi,V+\phi)$ is \compy\ inseparable, then so is $(U+\phi,V)$.
\end{thm}

\begin{proof}
The function $\psi \mapsto (\phi \wedge\psi)$ witnesses that \qedright
\[((U+\phi)_{\mf p},(V+\phi)_{\mf r})\leq_1 ((U+\phi)_{\mf p},V_{\mf r}).\]
\end{proof}


The \compu\ inseparability of a theory is closely related to the undecidability and the incompleteness of the theory. 
Indeed, the \compu\ inseparability of a theory $U$ implies the essential undecidability of $U$. 
For example, the \compu\ inseparability of the theory $\RR$ of weak arithmetic 
follows from the work by Smullyan~\cite{Smullyan58}, and then the essential undecidability of $\RR$ 
that was first established by Tarski, Mostowski and Robinson~\cite{TMR53} 
immediately follows. 
It is well-known that, for any consistent \ce~theory, essential incompleteness and essential undecidability 
are equivalent, and so the \compu\ inseparability of $\RR$ also yields the essential incompleteness of $\RR$. 
Here, the essential incompleteness of $\RR$ is also strengthened. 
Mostowski~\cite{Mos61} proved that $\RR$ is \textit{uniformly essentially incomplete}, that is, for any \compu\ sequence $(T_i)_{i \in \omega}$ of consistent \ce~extensions of $\RR$, there exists a sentence simultaneously independent of all theories in the sequence. 
Interestingly, Ehrenfeucht~\cite{ehre:sepa61} proved that Mostowski's theorem is equivalent to the \compu\ inseparability of $\RR$, namely, he
proved that, for any consistent \ce~theory $U$, $U$ is \compy\ inseparable if and only if $U$ is uniformly essentially incomplete. 

Finitely axiomatisable theories sometimes behave well. 
Tarski, Mostowski and Robinson~\cite{TMR53} showed that for a finitely axiomatisable theory, essential undecidability is equivalent to essential hereditary undecidability. 
Also, it follows from the Subtraction Theorem (Theorem~\ref{subtraction}) that for a finitely axiomatised theory, \compu\ inseparability is equivalent to strong \compu\ inseparability. 
Note that strong effective inseparability implies essential hereditary undecidability. 
So, the finitely axiomatised theory $\RQ$ which is an extension of $\RR$ is strongly \compy\ inseparable and essentially hereditarily undecidable. 
Here, since $\RR$ is not finitely axiomatisable, the essential hereditary undecidability of $\RR$ is non-trivial.
This was proved by Cobham, but his proof of the result was not published. 
Vaught~\cite{Vau62} gave a proof of Cobham's theorem by proving the strong \compu\ inseparability of $\RR$. 
For a detailed study of the notion of essential hereditary undecidability, see~\cite{Vis22}. 
See~\cite{Vis17} and~\cite{KV23} for new proofs of Cobham and Vaught's theorems.

Relating to these notions, we also introduce the following two notions: 
\begin{itemize}
    \item $U$ is \emph{f-essentially incomplete} iff, for any $U$-sentence $\phi$, if $U \cup \{\phi\}$ is consistent, then $U \cup \{\phi\}$ is incomplete. 
    
    \item $U$ is \emph{f-uniformly essentially incomplete} iff, for any $k \in \omega$, whenever
    $U_0$, \ldots, $U_{k-1}$ are consistent \ce~extensions of $U$ in the same language, then there is a sentence independent of each of the $U_0, \ldots, U_{k-1}$. 
\end{itemize}

It is easy to see that a theory $U$ is f-essentially incomplete iff the Lindenbaum algebra of $U$ is atomless. 
For f-uniform essential incompleteness, we have: 

\begin{prop}
Every essentially incomplete theory is f-uniformly essentially incomplete.
\end{prop}
\begin{proof}
We prove, by induction on $k$, that, if $U_0 \ldots, U_{k-1}$ are consistent \ce~extensions of $U$ in the same language, then there is a sentence $\rho_k$ independent of each of the $U_0, \ldots, U_{k-1}$.
We set $\rho_0 := \top$. 
In case $U_i \cup U_k$ is consistent for some $i<k$, we replace both $U_i$ and $U_k$ with $U_i \cup U_k$,
and apply the induction hypothesis to the reduced sequence. 
Suppose all the $U_i \cup U_k$, for $i<k$, are inconsistent.
We define $\rho_{k+1}$. For each $i<k$, there is a $\phi_i$, such that $U_k \vdash \phi_i$ and $U_i \vdash \neg\, \phi_i$.
Let $\phi$ be the conjunction of the $\phi_i$.
So, $U_k \vdash \phi$ and, for
each of the $U_i$, where $i<k$, we have $U_i \vdash \neg \,\phi$. 
Suppose $\rho$ is independent of $U_k$.
We define $\rho_{k+1} := (\rho \wedge \phi) \vee (\rho_k \wedge \neg\,\phi)$.
It is immediate that this does as promised.
\end{proof}

The relationships between these non-effective notions are visualised in Figure~\ref{Fig1}.
In~\cite{viss:nomi23}, the existence of a decidable f-essentially incomplete theory was proved. 
Also, essential hereditary undecidability and \compu\ inseparability are incomparable in general (cf.~\cite[Example 6]{Vis22}). 
Therefore, none of the implications in Figure~\ref{Fig1} are reversible.
Related to this figure, one could consider the notions such as f-essential undecidability and f-essential hereditary undecidability etc., but we will not deal with these notions, as they are beside the main subject of this paper. 

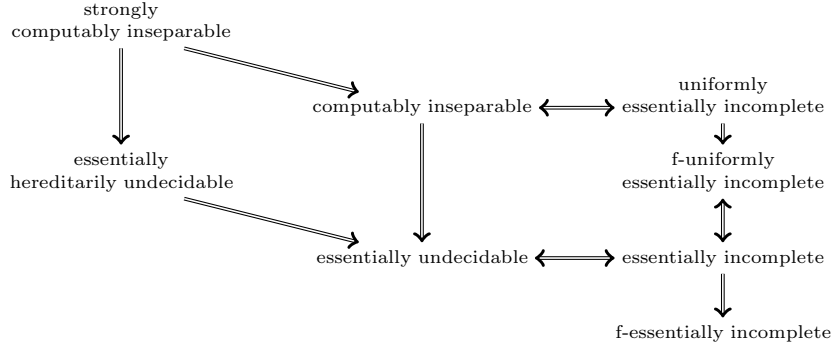
\begin{figure}[ht]
\centering
\begin{tikzpicture}
{\scriptsize
\node (fEI) at (4,-1) {f-essentially incomplete};

\node (EU) at (0,0) {essentially undecidable};

\node (EI) at (4,0) {essentially incomplete};

\node (WUEI1) at (4,1.3) {f-uniformly};
\node (WUEI2) at (4,1) {essentially incomplete};

\node (RI) at (0,2) {\compy\ inseparable};

\node (UEI1) at (4,2.3) {uniformly};
\node (UEI2) at (4,2) {essentially incomplete};

\node (EHU1) at (-4,1.3) {essentially};
\node (EHU2) at (-4, 1) {hereditarily undecidable};

\node (SRI1) at (-4,3.3) {strongly};
\node (SRI2) at (-4,3) {\compy\ inseparable};

\draw [<->, double] (EU)--(EI);
\draw [<->, double] (EI)--(WUEI2);
\draw [<->, double] (RI)--(UEI2);
\draw [->, double] (RI)--(EU);
\draw [->, double] (UEI2)--(WUEI1);
\draw [->, double] (EHU2)--(EU);
\draw [->, double] (SRI2)--(EHU1);
\draw [->, double] (SRI2)--(RI);
\draw [->, double] (EI)--(fEI);
}
\end{tikzpicture}
\caption{Implications between non-effective notions}\label{Fig1}
\end{figure}

In what follows, we explore the effectivisations of these notions of incompleteness, undecidability, and inseparability. 

\subsection{What is Effective?}\label{whatsmurf}
Notions like \emph{essential} and \emph{hereditary} operate extensionally on the
notions they modify. The situation is not so simple for adding \emph{effective}.
Adding ``effective'' in front of an expression operates intensionally on \emph{the definition} of the concept.

Suppose the definition of $\mc P$ has the form $\forall \vv x\,\exists \vv y\, \phi(\vv x,\vv y)$, 
where $\phi$ does not start with an existential quantifier. We propose to say that
 \emph{effectively $\mc P$} means that there are \compu\ functions $\vv \Phi$ 
 such that $\forall \vv x\, \phi(\vv x,\vv \Phi(\vv x))$.
If the given definition of $\mc P$ has the form $\forall \vv x\,(\psi(\vv x) \to \exists \vv y\, \phi(\vv x,\vv y))$,
where $\phi$ does not start with an existential quantifier,
 then, \emph{effectively $\mc P$} means that there are partial \compu\ functions $\vv \Phi$ such that
 \[\forall \vv x\, (\psi(\vv x) \to \exists \vv z\, (\vv \Phi(\vv x)\simeq \vv z \wedge  \phi(\vv x,\vv z))).\]

 \begin{rem}
     An alternative proposal would be to suggest that \emph{effectively $\mc P$} simply means the Kleene realisability
     of the salient definition of $\mc P$. However,  this does not always deliver the desired outcomes.
 \end{rem}
  
\subsubsection{Effective Undecidability}
  A \ce~set $\mc X$ is decidable iff $\exists i \, ((\mc X \cap {\sf W}_i) = \emptyset \wedge (\mc X \cup {\sf W}_i) = \omega)$.
So, \emph{$\mc X$ is undecidable} means:
\[\forall i\, (\exists y\, (y \in \mc X \wedge y\in {\sf W}_i) \vee \exists x\, (x\not\in \mc X \wedge x \not\in {\sf W}_i)).\]
Equivalently, 
\[\forall i\, \exists x\, ((x \in \mc X \wedge x\in {\sf W}_i) \vee (x\not\in \mc X \wedge x \not\in {\sf W}_i)).\] 
The constructivisation of this is: there is a \compu\ $\Phi$ such that:
\[\forall i\,  ((\Phi(i) \in \mc X \wedge \Phi(i)\in {\sf W}_i) \vee (\Phi(i)\not\in \mc X \wedge \Phi(i) \not\in {\sf W}_i)).\] 
So this is the notion of \emph{being constructively non-\compu}. See \cite[p162]{roge:theo67}. 

Alternatively, \emph{$\mc X$ is undecidable} also means:
\[\forall i\, \exists y\, ((\mc X \cap {\sf W}_i) = \emptyset \Rightarrow (y \not \in \mc X \wedge y \not \in {\sf W}_i)).\]
So, the effectivisation of this is: there is a \compu\ $\Phi$ such that: 
\[\forall i\,  ((\mc X \cap {\sf W}_i) = \emptyset \Rightarrow  (\Phi(i) \not \in \mc X \wedge \Phi(i) \not \in {\sf W}_i)).\] 
This is exactly the notion of \emph{being creative}. 

Every constructively non-\compu\ set is exactly a \ce~set whose complement is completely productive, which is a notion introduced by Dekker~\cite{Dek55}. 
It is proved in \cite[p183, Theorem VI]{roge:theo67} that productivity and complete productivity coincide, and hence creativity and constructive non-\compab\ also coincide. 
So, these notions serve stable effectivisation of the notion of undecidability.

\subsubsection{Effective Essential Undecidability}
Let us assume  the definition of the essential undecidability of $U$ is:
 \begin{multline*}
\forall i\,\forall j \, \Bigl(\bigl({\sf W}_{i} \vdash U\text{ and } {\sf con}({\sf W}_i)\bigr)\; \To \\
  \exists x\, \bigl((x\not\in {\sf W}_{i\mf p} \wedge x \not\in {\sf W}_j) \vee  (x \in {\sf W}_{i\mf p} \wedge x\in {\sf W}_j)\bigr)\Bigr). 
\end{multline*}
Here ${\sf con}({\sf W}_i)$ is an abbreviation of the statement `${\sf W}_i$ is consistent'. 
So our recipe gives:  there is a partial \compu\ $\Psi$ such that:
{\small
\begin{multline*}
  \forall i\,\forall j \, \Bigl(\bigl({\sf W}_{i} \vdash U\text{ and } {\sf con}({\sf W}_i)\bigr) \To \\
 \bigl(\Psi(i,j) {\downarrow} \wedge
   \bigl((\Psi(i,j)\not\in {\sf W}_{i\mf p} \wedge \Psi(i,j) \not\in {\sf W}_j) \vee  
   (\Psi(i,j) \in {\sf W}_{i\mf p} \wedge \Psi(i,j)\in {\sf W}_j)\bigl)\bigl)\Bigr).
\end{multline*}
} 
By the usual argument, we can always choose $\Psi$ total.
Moreover, our definition is equivalent to:
there is a total \compu\ $\Theta$ such that:
{\small
 \[ \forall i \, \bigl(({\sf W}_{i} \vdash U\text{ and } {\sf con}({\sf W}_i)\bigr) \To 
\bigl( \lambda j.\Theta(i,j) \text{ witnesses that ${\sf W}_{i\mf p}$ is creative)}\bigr) \] 
}
So this gives us that effective essential undecidability is the same thing as effective essential creativity.
We will work with the last notion.
This notion was suggested in Feferman's paper \cite[Footnote 11]{fefe:degr57} and investigated by Smullyan~\cite{Smullyan60}.

In the rest of the paper, we will simply stipulate the effective versions of the relevant notions. The reader may amuse
herself by deriving the definitions following our recipe. We briefly discuss why there is not separate notion of
effective local interpretability in Appendix~\ref{locosmurf}.

\subsubsection{Constraining the Witness}
Effective notions usually have a partial \compu\ function $\Phi$ that chooses some (counter)example.
In many cases, it is interesting to put a constraint on the (counter)examples, i.e., on the range of
witness providing function $\Phi$. For example, consider effective essential incompleteness.
One usually specifies that the witnesses can be chosen to be $\Sigma^0_1$ (or, equivalently, $\Pi^0_1$). 
In this case we will speak of effective essential $\Sigma^0_1$-incompleteness. 
We will see that there are other interesting restriction than this familiar one.

More generally, if the constraint
$\mc X$ is a set of numbers coding sentences-of-the-given-signature, and $\mc P$ is the property of theories under
consideration, we will speak about \emph{effective $\mc X$-$\mc P$}.
Note that we do not demand that $\mc X$ is \ce.

We can make this even more general. Let $\mc F$ be a function from sets of sentences to sets of sentences (all of the given signature).
We do not put any effectivity constrains on $\mc F$. Moreover, we allow $\mc F$ to be empty on some arguments.
For example, $U$ is effectively essentially $\mc F$-incomplete iff, for every $i$ such that ${\sf W}_i$ axiomatizes a consistent extension of
$U$, there is a $\phi\in \mc F({\sf W}_i)$, such that $\phi$ is independent of ${\sf W}_i$.  
We note that if $\mc F$ has constant value $\mc X$, we are back in the simpler case.

\subsection{Effective Inseparability}
Two disjoint \ce~sets $\mc X$ and $\mc Y$ are said to be \emph{effectively inseparable} iff, there exists a partial \compu\ function $\Phi$ such that for any \ce~sets ${\sf W}_i$ and ${\sf W}_j$, if ${\sf W}_i$ and ${\sf W}_j$ weakly bi-separate $\mc X$ and $\mc Y$, then $\Phi(i, j)$ converges and $\Phi(i, j) \notin {\sf W}_i \cup {\sf W}_j$. 
Let $(U,V)$ be acceptable pair of theories. We define:
\begin{itemize}
    \item $(U,V)$ is \emph{effectively inseparable} iff $U_{\mf p}$ and $V_{\mf r}$ are effectively inseparable.
    \item $U$ is \emph{effectively inseparable} iff $(U,U)$ is effectively inseparable.
    \item $U$ is \emph{strongly effectively inseparable} iff $(U, 0_U)$ is effectively inseparable.
\end{itemize}

We define witness comparison notation.
For every \ce~relation $R(\vec{x})$, we can effectively find a primitive \compu\ relation $R^\star(\vec{x}, y)$ such that $R(\vec{x})$ iff $\exists y\, R^\star(\vec{x}, y)$. 
For all pairs of \ce~relations $R_0(\vec{x})$ and $R_1(\vec{x})$, we define:
\begin{itemize}
    \item $R_0(\vec{x}) \leq R_1(\vec{x}) :\iff \exists y\, (R_0^\star(\vec{x}, y) \wedge \forall z < y\, \neg\,R_1^\star(\vec{x}, z))$,
    \item $R_0(\vec{x}) < R_1(\vec{x}) :\iff \exists y\, (R_0^\star(\vec{x}, y) \wedge \forall z \leq y\, \neg\,R_1^\star(\vec{x}, z))$. 
\end{itemize}
We note that the witness comparison notation is \emph{intensional}. The procedure to
find $R^\star$ given $R$ operates on a presentation of $R$ and gives a presentation of $R^\star$ as
output. 

\begin{prop}
In the definition of effective inseparability, restricting weakly bi-separating sets to bi-separating sets yields an equivalent notion. 
\end{prop}
\begin{proof}
This is because if ${\sf W}_i$ and ${\sf W}_j$ weakly bi-separate $\mc X$ and $\mc Y$, then we can effectively find $k_0$ and $k_1$ such that ${\sf W}_{k_0} \cup {\sf W}_{k_1} = {\sf W}_i \cup {\sf W}_j$, and ${\sf W}_{k_0}$ and ${\sf W}_{k_1}$ bi-separate $\mc X$ and $\mc Y$. 
Such $k_0$ and $k_1$ are obtained by letting
\begin{itemize}
    \item ${\sf W}_{k_0} = \{n \mid (n \in {\sf W}_i) \leq (n \in {\sf W}_j)\}$ and 
    \item ${\sf W}_{k_1} = \{n \mid (n \in {\sf W}_j) < (n \in {\sf W}_i)\}$.\qedhere
\end{itemize}
\end{proof}

Smullyan analyzed the notion of effective inseparability extensively, and in particular various notions related to effective inseparability are discussed in his~\cite{Smullyan63,Smullyan93}. 
For more information on this topic, see the recent paper by Yong Cheng \cite{Cheng23}.

The effective inseparability of the theory $\RR$ follows from a result established by Smullyan~\cite{Smullyan60}. 
Smullyan mentioned that effective inseparability implies effective essential creativity. 
Also, for any consistent \ce~theory, effective essential creativity clearly implies effective essential incompleteness. 
Ehrenfeucht~\cite{ehre:sepa61} provided an essentially incomplete theory that is not \compy\ inseparable, but, in the effective case, such an example cannot exist. 
Namely, for any consistent \ce~theory, effective essential incompleteness implies effective inseparability. 
This result was established by Pour-El. 
Actually, Pour-El proved more. 
We say that a theory $U$ is \emph{effectively if-essentially $\mc F$-incomplete} iff there is a partial \compu\ function $\Phi$ such that for any $i$, if ${\sf W}_i$ is a consistent finite extension of $U$, then $\Phi(i)$ converges, $\Phi(i) \in \mc F({\sf W}_i)$, and $\Phi(i)$ is independent of ${\sf W}_i$. 
Here, `if' stands for `intensional finite extensions'. 
The notion of effective ef-essential $\mc F$-incompleteness is studied in Section~\ref{sec_ef}, where `ef' stands for `extensional finite extensions'. 
Pour-El called effective essential incompleteness \emph{effective extensibility}. 
She called effective if-essential incompleteness \emph{weak effective extensibility}.
Pour-El's theorem is stated as follows: 

\begin{thm}[Pour-El~{\cite[Theorem 1]{pour:effe68}}]
For any consistent \ce~theory $U$, the following are equivalent: 
\begin{enumerate}[a.]
    \item $U$ is effectively inseparable.
    \item $U$ is effectively essentially creative.
    \item $U$ is effectively essentially incomplete.
    \item $U$ is effectively if-essentially incomplete.
\end{enumerate}
\end{thm}

In the next section, we prove a slightly strengthened version of Pour-El's theorem (Theorem~\ref{pour-el}). 
Furthermore, our proof is slightly simpler than the original one.
For completeness, we describe Pour-El's original argument in Appendix~\ref{original}.

We say that a theory $U$ is \emph{effectively uniformly essentially $\mc X$-incomplete} iff, there is a partial \compu\ function $\Phi$ such that for any $j$, if $j$ is the index of a \compu\ sequence of consistent \ce~extensions $(T_i)_{i \in \omega}$ of $U$, then we have that $\Phi(j)$ converges, $\Phi(j)\in \mc X$, and $\Phi(j)$ is independent of $T_i$ for all $i \in \omega$. 
For any consistent \ce~theory, effective inseparability easily implies effective uniform essential incompleteness. 
So, one could say that, in the effective case, Ehrenfeucht's result on the equivalence between \compu\ inseparability and uniform essential incompleteness is superseded by Pour-El's work.
However, we do think it is instructive to
include the effective analogues of Ehrenfeucht's theorem here, also since these results have entirely self-reference free proofs.  
We also present the effective version of Ehrenfeucht's results from the mono-perspective.

\begin{thm}\label{effehr3}
Let $U$ be a consistent \ce~theory. Let $\mc X$ be a set of sentences.
The following are equivalent.
\begin{enumerate}[a.]
\item
$U$ is effectively $\mc X$-inseparable, i.e., there is a partial \compu\ function $\Phi$, such that for all pairs of 
sets ${\sf W}_i$, ${\sf W}_j$ that weakly bi-separate $U_{\mf p}$ and $U_{\mf r}$, 
we have $\Phi(i,j)$ converges, $\Phi(i,j) \in \mc X$, and $\Phi(i,j) \not\in \mathsf W_i \cup \mathsf W_j$. 
\item
$U$ is effectively uniformly essentially $\mc X$-incomplete, i.e., there is a partial \compu\ function $\Psi_0$, such that, 
for every \compu\ sequence of consistent \ce~extensions $U'_i$ of $U$ with index $j$, we have that
$\Psi_0(j)$ converges, $\Psi_0(j)\in \mc X$, and, for all $i$,  $U'_i \nvdash  \Psi_0(j)$ and $U'_i \nvdash \neg\, \Psi_0(j)$.
\item
There is a partial \compu\ function $\Psi_1$, such that, 
for every \compu\ sequence of $U$-sentences $\nu_0,\nu_1,\dots$ with index $j$, such that each $\nu_i$ is consistent with $U$, 
$\Psi_1(j)$ converges, $\Psi_1(j) \in \mc X$, and,
 for all $i$, we have $U \nvdash \nu_i \to \Psi_1(j) $ and $U \nvdash \nu_i \to \neg\, \Psi_1(j)$.
\item
There is a partial \compu\ function $\Psi_2$, such that, for every \compu\ sequence of $U$-sentences $\chi_0,\chi_1,\dots$ with index $j$, 
such that each $\chi_i$ is consistent with $U$, we have $\Psi_2(j)$ converges,  $\Psi_2(j)\in \mc X$, and,
for all $i$, we have  $0_U \nvdash \chi_i \to \Psi_2(j)$ and $0_U \nvdash \chi_i \to \neg\, \Psi_2(j)$.
\item
$U$ is effectively $\Subset$-essentially $\mc X$-$\mf m$-incomplete, i.e.,
there is a partial \compu\ function $\Psi_3$, such that for every mono-consistent  ${\sf W}_j\Supset U$, we have
$\Psi_3(j)$ converges, $\Psi_3(j)\in \mc X$, and, $ \Psi_3(j)\not\in {\sf W}_{j\mf m}\cup {\sf W}_{j\mf n}$.
\item
There is a partial \compu\ function $\Psi_4$ such that, for  any ${\sf W}_j$ such that $\widehat U \Cup {\sf W}_j$ is mono-consistent,
we have $\Psi_4(j)$ converges, $\Psi_4(j)\in \mc X$, and, $ \Psi_4(j)\not\in {\sf W}_{j\mf m}\cup {\sf W}_{j\mf n}$.
\end{enumerate}

Moreover, each of \textup(a\textup), \textup(b\textup), \textup(c\textup), \textup(d\textup), \textup(e\textup), \textup(f\textup) is equivalent to a version, say \textup(a$'$\textup),  
\textup(b$'$\textup), \textup(c$'$\textup), \textup(d$'$\textup), \textup(e$'$\textup),  \textup(f$'$\textup) where the witnessing function is total.
\end{thm}
\begin{proof}
``(a) to (b)''. 
Consider a \compu\ sequence $(U'_i)_{i\in \omega}$ of consistent \ce~extensions of $U$. Say our sequence has index $j$.
Let $V:= \bigcup_{i\in \omega} U'_{i\mf p}$ and let $W := \bigcup_{i\in \omega} U'_{i\mf r}$.
We can find indices $k$ and $\ell$ for $V$ and $W$ effectively from $j$. 
Clearly, $V$ and $W$ weakly bi-separate $U_{\mf p}$ and $U_{\mf r}$.   
We take $ \Psi_0(j) := \Phi(k,\ell)$.
It follows that $ \Psi_0(j)$ is in $\mc X$ and is independent of each $U'_i$.
 
\medskip
``(b) to (c)''. We can effectively transform an index $j$ of the sequence $(\nu_i)_{i\in \omega}$ into an index $j'$ of the sequence of theories
$(U+\nu_i)_{i\in \omega}$. We set $\Psi_1(j) := \Psi_0(j')$.

\medskip
``(c) to (d)''. We can take $\Psi_2:= \Psi_1$.

\medskip
``(d) to (e)''. Suppose ${\sf W}_j \Supset U$ and ${\sf W}_j$ is mono-consistent. 
We can effectively find an index $k$ of an enumeration $(\chi_s)_{s\in \omega}$ of the elements of ${\sf W}_j$ from $j$.
Since ${\sf W}_j$ and $U_{\mf r}$ are disjoint, we have that $\chi_s$ is consistent with $U$ for each $s \in \omega$. 
We obtain that $\Psi_2(k)$ converges, $\Psi_2(k) \in \mc X$, and for all $s$, $0_U \nvdash \chi_s \to \Psi_2(k)$ and $0_U \nvdash \chi_s \to \neg\, \Psi_2(k)$. 
We then have $\Psi_2(k) \notin {\sf W}_{j \mf m} \cup {\sf W}_{j \mf n}$. 
We take $\Psi_3(j) := \Psi_2(k)$.

\medskip
``(e) to (f)''. Consider any \ce~${\sf W}_j$ such that $W:=\widehat U \Cup {\sf W}_j$ is mono-consistent.
It is easy to see that $W \Supset U$. We can easily find an index $k$ of $W$ from $j$. We take $\Psi_4(j) := \Psi_3(k)$.

\medskip
``(f) to (a)''. Suppose ${\sf W}_i$ and ${\sf W}_j$ weakly bi-separate $U_{\mf p}$ and $U_{\mf r}$.
Let \[ V := {\sf W}_i \cup \verz{ \chi \mid (\neg\,\chi) \in {\sf W}_j}.\] 
If $\widehat U \Cup V$ were mono-inconsistent, there would be a $\nu \in V$ and $U \vdash \neg\, \nu$. The second conjunct tells us
that $ \nu \not\in {\sf W}_i$ and $(\neg\,\nu )\not\in {\sf W}_j$. A contradiction. So,
$\widehat U \Cup V$ is mono-consistent. We clearly can find an index $k$ of $V$ effectively from $i$ and $j$.
We take $\Phi(i,j) := \Psi_4(k)$.
 
\medskip
 We prove the equivalence between (a) and (a$'$). The (a$'$)-to-(a) direction is trivial. We assume (a).
 Consider any pair of indices $i,j$. We define $\Phi^\ast$.
We can effectively find indices $i',j'$ of $U_{\mf p} \cup {\sf W}_i$ and $U_{\mf r} \cup {\sf W}_j$.
 We compute $\Phi(i',j')$ and, simultaneously, we seek a counterexample to the claim that ${\sf W}_{i'},{\sf W}_{j'}$ weakly
 biseparate ${\sf U}_{\mf p}$ and $U_{\mf r}$. If we find a value of $\Phi(i',j')$ first then we give that as output of $\Phi^\ast(i,j)$.
 If we find a counterexample first, we give that as output (or, alternatively, we give some fixed chosen sentence as output). 
 It is easy to see that $\Phi^\ast$ is total and satisfies our specification.
 
 The equivalences of (b) and (b$'$), (c) and (c$'$), (d) and (d$'$), (e) and (e$'$), and (f) and (f$'$) are proved in an analogous way. 
\end{proof}

\section{Pour-El's Theorem}\label{pour-el-sec}

In this section, we prove Pour-El's Theorem in a slightly stronger version (Section~\ref{mainsmurf}). 
We give a variant (Section~\ref{doublesmurf}) and describe some sample applications
(Sections~\ref{orey} and~\ref{conssmurf}). 
In Section~\ref{heredity}, we present an adaptation of the argument that applies to effective hereditary creativity.

\subsection{The Theorem}\label{mainsmurf}

In this subsection, we give our version of Pour-El's result.

\begin{thm}\label{pour-el}
    Suppose $U$ is effectively if-essentially $\mc F$-incomplete. Then, $U$ is effectively $\mc F(U)$-inseparable.
\end{thm}

\begin{proof}
Let $U$ be effectively if-essentially $\mc F$-incomplete with  partial \compu\ witness $\Phi$.
Suppose ${\sf W}_i$ and ${\sf W}_j$ weakly biseparate $U_{\mf p}$ and $U_{\mf r}$.

Using the Recursion Theorem with parameters (cf.~\cite[Theorem 3.5]{Soar87}), we effectively find a $k^\ast$ (depending of $i$ and $j$) such that
\begin{multline*}
{\sf W}_{k^\ast} = U \cup \verz{\phi \mid  \Phi(k^\ast)\simeq \phi \text{ and }\phi \in {\sf W}_i}\; \cup \\
  \verz{\psi \mid \exists \phi \,(\Phi(k^\ast)\simeq \phi \text{ and }\phi \in {\sf W}_j  \text{ and } \psi = \neg\,\phi)}.
  \end{multline*}
We note that we have:
{\small
\[
	{\sf W}_{k^\ast} = \begin{cases} U \cup \{\Phi(k^\ast)\}  & \ \text{if $\Phi(k^\ast)$ converges and
      $\Phi(k^\ast) \in {\sf W}_i$},\\
	U \cup \{\neg\, \Phi(k^\ast)\}  & \ \ \text{if $\Phi(k^\ast)$ converges and
      $\Phi(k^\ast) \in {\sf W}_j$}, \\
	U & \ \text{otherwise}. \end{cases}
\]
}%

Let $U^\ast :={\sf W}_{k^\ast}$.
\begin{itemize}
    \item 
Suppose $\Phi(k^\ast)$ converges to, say, $\phi^\ast$ and 
$\phi^\ast \in {\sf W}_i$. It follows that  $U^\ast= U\cup\verz{\phi^\ast}$. 
Since $\phi^\ast\not\in U_{\mf r}$, we have that $U^\ast$ is consistent and, hence, $\phi^\ast$ is independent of $U^\ast$. A contradiction.
\item
Suppose $\Phi(k^\ast)$ converges to, say, $\phi^\ast$ and 
$\phi^\ast \in {\sf W}_j$. It follows that $U^\ast= U\cup\verz{\neg\,\phi^\ast}$.
Since $\phi^\ast\not\in U_{\mf p}$, we have that $U^\ast$ is consistent, and, hence, $\phi^\ast$ is independent of $U^\ast$. A contradiction.
\end{itemize}
We may conclude that $U^\ast = U$. Hence, $\Phi(k^\ast)$ converges, say to $\phi^\ast$. We have $\phi^\ast \not\in {\sf W}_i\cup{\sf W}_j$.

Clearly, our argument delivers a total \compu\ $\Psi$, such that, whenever ${\sf W}_i$ and ${\sf W}_j$ weakly biseparate $U_{\mf p}$ and $U_{\mf r}$,
we have $\Psi(i,j) \not\in {\sf W}_i\cup{\sf W}_j$. 

Finally, we note that the ${\sf W}_{k^\ast}$ are all equal to $U$ and, thus $\Psi(i,j)= \Phi(k^\ast) \in \mc F(U)$. 
\end{proof}

As a consequence of Theorem~\ref{pour-el}, we obtain the following corollary showing the equivalence of several relating notions. 
For each formula class $\mc X$ and theory $U$, let $[\mc X]_U$ be the closure of $\mc X$ under $U$-provable equivalence.


\begin{cor}\label{cor_pour-el}
For any consistent \ce~theory $U$ and set of sentences $\mc X$, the following are equivalent: 
\begin{enumerate}[a.]
    \item $U$ is effectively $\mc X$-inseparable.
    \item $U$ is effectively essentially $\mc X$-creative.
    \item $U$ is effectively essentially $\mc X$-incomplete.
    \item $U$ is effectively if-essentially $\mc X$-incomplete. 
\end{enumerate}

Moreover, in case $\mc X$ is \ce, each of \textup(a\textup), \textup(b\textup), \textup(c\textup), \textup(d\textup) is equivalent to a version, 
say \textup(a$'$\textup), \textup(b$'$\textup), \textup(c$'$\textup), \textup(d$\,'$\textup) where $\mc X$ is replaced by $[\mc X]_U$.
\end{cor}
\begin{proof}
The implications ``(a) to (b)'', ``(b) to (c)'', and ``(c) to (d)'' are obvious. 

\medskip
``(d) to (a)''. Immediate from Theorem~\ref{pour-el} by letting $\mc F$ to be the function having the constant value $\mc X$. 

\medskip
Suppose $\mc X$ is \ce
It suffices to show the equivalence of (c) and \textup(c$'$\textup). The implication ``(c) to \textup(c$'$\textup)'' is trivial because $\mc X \subseteq [\mc X]_U$. 
We treat ``\textup(c$'$\textup) to (c)''. 
Suppose that $U$ is effectively essentially $[\mc X]_U$-incomplete, as witnessed by a partial \compu\ function $\Phi$. 
Note that if $T$ is a consistent extension of $U$ and $\phi$ is independent of $T$, then each element of $[\{\phi\}]_U$ is also independent of $T$. 
Let $\Psi$ be a partial \compu\ function such that for each $i$, if $\Phi(i)$ converges, then $\Psi(i)$ is in $[\{\Phi(i)\}]_U \cap \mc X$. 
Then, it is shown that $\Psi$ witnesses the effective $\mc X$-incompleteness of $U$. 
\end{proof}

\begin{cor}\label{cor_pour-el_2}
If $U$ is effectively if-essentially $\mc F$-incomplete, then $U$ is effectively $\mc F(U)$-inseparable and effectively if-essentially $\mc F(U)$-incomplete
\end{cor}

\subsection{A Variant: Double Generativity}\label{doublesmurf}

In Smullyan's book \cite{Smullyan93}, many notions that are equivalent to effective inseparability were introduced (see also \cite{Cheng23}).
As a sample of variants of the argument in Subsection \ref{mainsmurf}, we focus on double generativity among them as the double analogue of constructive non--\compab, and discuss its witness constraining version. 

We say that a theory $U$ is \textit{doubly $\mc X$-generative} iff there exists a total \compu\ function $\Phi$ such that for any $i, j \in \omega$, if ${\sf W}_i \cap {\sf W}_j = \emptyset$, then \begin{itemize}
    \item $\Phi(i, j) \in U_{\mf p}$ iff $\Phi(i, j) \in {\sf W}_j$, 
    \item $\Phi(i, j) \in U_{\mf r}$ iff $\Phi(i, j) \in {\sf W}_i$, 
    \item if $\Phi(i, j) \notin {\sf W}_i \cup {\sf W}_j$, then $\Phi(i, j) \in \mc X$. 
\end{itemize}

We obtain a \emph{prima facie} less constrictive notion if we demand that the ${\sf W}_i$, ${\sf W}_j$ are the $V_{\mf p}$, $V_{\mf r}$ of some theory $V$. 
A consistent \ce~theory $U$ is \emph{$\mc X$-theory-generative} iff, there is a total 
\compu\ function $\Psi$, such for every consistent theory $V$ in the
$U$-language with index $i$, we have:
\begin{itemize}
    \item $\Psi(i) \in U_{\mf p}$ iff $\Psi(i) \in V_{\mf r}$, 
    \item $\Psi(i) \in U_{\mf r}$ iff $\Psi(i) \in V_{\mf p}$, 
    \item if $\Psi(i) \notin V_{\mf p} \cup V_{\mf r}$, then $\Psi(i) \in \mc X$.
    \end{itemize}

\begin{thm}
For any consistent \ce~theory $U$ and set of sentences $\mc X$, the following are equivalent: 
\begin{enumerate}[a.]
    \item $U$ is doubly $\mc X$-generative.
    \item $U$ is $\mc X$-theory-generative.
    \item $U$ is effectively if-essentially $\mc X$-incomplete. 
\end{enumerate}
\end{thm}
\begin{proof}
``(a) to (b)''. Suppose that a total \compu\ function $\Phi$ witnesses the double $\mc X$-generativity of $U$. 
Let $V$ be any consistent \ce~theory  with index $i$. 
We can effectively find $k_0$ and $k_1$ from $i$ such that ${\sf W}_{k_0} =V_{\mf p}$ and ${\sf W}_{k_1} = V_{\mf r}$. Let $\Psi(i) := \Phi(k_0,k_1)$.
It is immediate that $\Phi$ witnesses that $U$ is $\mc X$-theory-generative.

\medskip
``(b) to (c)''. Suppose $\Phi$ witnesses that $U$ is $\mc X$-theory-generative and let $V$ be a consistent \ce~theory that if-extends $U$ with index $i$.
Since the first two cases of $\mc X$-theory-generativity cannot be active, it follows that 
$\Phi(i)\not\in V_{\mf p} \cup V_{\mf r}$ and
$\Phi(i) \in \mc X$.

\medskip
``(c) to (a)''. Suppose that $U$ is effectively if-essentially $\mc X$-incomplete. 
Let $\Phi$ be a partial \compu\ function that witnesses the effective if-essential $\mc X$-incompleteness of $U$. 
We may assume that $\Phi(k)$ converges if ${\sf W}_k$ is an if-extension of $U$, whether ${\sf W}_k$ is consistent or not. 
By the Recursion Theorem with parameters, there exists a \compu\ function $\Psi(x, y)$ such that, setting $\Phi^*(x, y) : = \Phi(\Psi(x, y))$, we have:
{\small
\[
	{\sf W}_{\Psi(x, y)} = \begin{cases} U \cup \{\Phi^*(x, y)\}  & \ \text{if}\ \bigl(\Phi^*(x, y) \in (U_{\mf p} \cup {\sf W}_{x}) \bigr) \leq \bigl(\Phi^*(x, y) \in (U_{\mf r} \cup {\sf W}_{y}) \bigr),\\
	U \cup \{\neg \,\Phi^*(x, y)\}  & \ \text{if}\ \bigl(\Phi^*(x, y) \in (U_{\mf r} \cup {\sf W}_{y}) \bigr) < \bigl(\Phi^*(x, y) \in (U_{\mf p} \cup {\sf W}_{x}) \bigr), \\
	U & \ \text{otherwise}. \end{cases}
\]
}%
Since ${\sf W}_{\Psi(i, j)}$ is an if-essential extension of $U$ for all $i$ and $j$, we find that $\Phi^*(i, j)$ is a total \compu\ function. 

We show that $\Phi^*$ witnesses the double $\mc X$-generativity of $U$. 
Let $i$ and $j$ be such that ${\sf W}_i \cap {\sf W}_j = \emptyset$ and let $\phi^\ast := \Phi^\ast(i,j)$ 

\begin{itemize}
	\item Suppose $\bigl(\phi^* \in (U_{\mf p} \cup {\sf W}_{i}) \bigr) \leq \bigl(\phi^* \in (U_{\mf r} \cup {\sf W}_{j}) \bigr)$ holds. 
    Then, ${\sf W}_{\Psi(i, j)} = U \cup \{\phi^*\}$. 
 Since ${\sf W}_{\Psi(i, j)} \vdash \Phi(\Psi(i, j))$, we have that ${\sf W}_{\Psi(i, j)} = U \cup \{\phi^*\}$ is inconsistent. 
So, $\phi^* \in U_{\mf r}$. 
Since $U_{\mf p} \cap U_{\mf r} 
= \emptyset$ and $\phi^* \in U_{\mf p} \cup {\sf W}_i$, we obtain 
$\phi^* \in U_{\mf r} \cap {\sf W}_i$. 
 
 	\item Suppose $\bigl(\phi^* \in (U_{\mf r} \cup {\sf W}_{j}) \bigr) < \bigl(\phi^* \in (U_{\mf p} \cup {\sf W}_{i}) \bigr)$ holds. 
    As above, it is shown that $\phi^* \in U_{\mf p} \cap {\sf W}_j$. 

    \item Otherwise, $\phi^* \notin U_{\mf p} \cup {\sf W}_j$ and $\phi^* \notin U_{\mf r} \cup {\sf W}_i$. 
    Since ${\sf W}_{\Psi(i, j)} = U$ is a consistent if-essential extension of $U$, we have $\phi^* = \Phi(\Psi(i, j)) \in \mc X$. 
\end{itemize}
A simple exercise in propositional logic now  shows  that $\Phi^*$ witnesses the double $\mc X$-generativity of $U$.
\end{proof}

\subsection{Orey-Sentences of Extensions of Peano Arithmetic}\label{orey}

In this subsection we treat Orey-sentences for extensions of {\sf PA}. The results of Section~\ref{conssmurf} will extend these results, but
it is nice to see the simple case first.

Let $U$ be any consistent \ce~extension of Peano Arithmetic. An \emph{Orey-sentence} of $U$ is a sentence $O$, such that
$U \rhd (U+O)$ and $U \rhd (U+ \neg\, O)$. Clearly, an Orey sentence $O$ of $U$ is independent of $U$.
We also note that the Orey property is extensional in the sense that it only depends on the theorems of the
given theory.

Here is one way  to construct an Orey-sentence for $U$. Let ${\sf W}_i$ be an enumeration of the axioms of $U$.
We define ${\sf W}_{i,n}$ as the theory axiomatised by the first $n$ axioms enumerated in ${\sf W}_i$. We define
$\apr_i \phi :\iff \exists x\, (\opr_{{\sf W}_{i,x}} \phi \wedge \oco_{{\sf W}_{i,x}} \top)$. Here $\opr_{{\sf W}_{i,x}}$ stands for provability in ${\sf W}_{i,x}$ and 
$\oco$ stands for $\neg\opr\neg$, so, $\oco_{W_{i,n}}\top$ arithmetizes the consistency of ${\sf W}_{i,n}$.
We write $\aco$ for $\neg\apr\neg$. We find that $\apr_i$ satisfies the modal laws of {\sf K}. Moreover,
we have the seriality axiom {\sf D}, i.e. $\aco_i\top$. Finally, 
 we can prove $\aco_i\phi\rhd_U \phi$. See~\cite{fefe:arit60} or~\cite{viss:inte18}.

We find $\gamma_i$ with ${\sf PA} \vdash \gamma_i \iff \neg\, \apr_i \gamma_i$.
We claim that $\gamma_i$ is an Orey-sentence of $U$.
We have, temporarily omitting the subscripts $i$ and $U$:
\begin{eqnarray*}
     \gamma & \vdash &  \neg\,\apr \gamma \\
    & \vdash & \aco \neg\, \gamma \\
    & \rhd & \neg\, \gamma \\
    \neg \, \gamma & \vdash & \apr \gamma \\
    & \vdash & \aco \gamma \\
    & \rhd & \gamma
\end{eqnarray*}
Since, we have $\gamma \rhd \neg\, \gamma$ and $\neg\,\gamma \rhd \neg\, \gamma$, we find, using a disjunctive interpretation,
$\top \rhd \neg\, \gamma$. Similarly, we find $\top \rhd  \gamma$. Thus, $\gamma_i$ is an Orey-sentence of $U$ as promised.

\begin{rem}
    The Orey-sentences $\gamma_i$ produced in our example are all known to be true. We can also use $\apr_i$ to build
    $U$-internally a Henkin interpretation of $U$. This interpretation comes with a truth predicate ${\sf H}_i$.
    The Liar sentence $\lambda_i$ of ${\sf H}_i$ is also an Orey-sentence of $U$. However, in this case, it is unknown whether
    $\lambda_i$ or $\neg\,\lambda_i$ is true. 

    We note that, in the real world, the Henkin construction just depends on the theorems of the given theory, not on the
    axiomatisation. It \emph{does} depend on the chosen enumeration of sentences. Thus, as long as we keep the enumeration fixed,
    the truth-value of $\lambda_i$ remains the same when we run through $i$ that enumerates axiom sets of {\sf PA}.
\end{rem}

Let $\mathbb O$ be the function that assigns to sets of sentences $\mc A$ in the signature of arithmetic the Orey-sentences of the theory axiomatised
by $\mc A$. Note that it is possible that there  are no such Orey-sentences, so the empty set will be in the range of this function. 
We have shown: 
\begin{thm}
  ${\sf PA}$ is effectively essentially $\mathbb O$-incomplete.  
\end{thm}

Applying Corollary~\ref{cor_pour-el_2}, we find:

\begin{thm}\label{lolsmurf}
    {\sf PA} is effectively $\mathbb O({\sf PA})$-inseparable and, thus,
effectively essentially $\mathbb O({\sf PA})$-incomplete.
\end{thm}

Theorem~\ref{lolsmurf} illustrates the important insight that \emph{independence, \emph{per se}, has nothing to do with strength}.
Of course, we are familiar with this point e.g. from the well-known results on, e.g., the continuum hypothesis which is an Orey-sentence
of {\sf ZF} that is independent of many extensions that have to do with strength. However, Theorem~\ref{lolsmurf} has a somewhat different
flavor in that it presents a systematic construction of such sentences for a wide range of theories.

\subsection{Consistency and Conservativity}\label{conssmurf}

We say a formula $\alpha(x)$ is a \textit{binumeration} of a theory $U$ iff for any $U$-sentence $\varphi$, if $\varphi \in U$, then $U \vdash \alpha(\gn{\varphi})$; and if $\varphi \notin U$, then $U \vdash \neg\, \alpha(\gn{\varphi})$. 
For each binumeration $\alpha(x)$ of $U$, we can naturally construct a formula $\Prf_{\alpha}(x, y)$ saying that $y$ is the code of a proof of a formula with the code $x$ in the theory defined by $\alpha$ (cf.~\cite{fefe:arit60}). 
Let $\Con_\alpha$ be the consistency statement $\neg\, \exists y\, \Prf_\alpha(\gn{\bot}, y)$ of $U$, where $\bot$ is some $U$-refutable sentence. 
Let $\mathbb C$ be the function that assigns to sets of sentences $\mc A$ in the signature of arithmetic 
the set of all sentences of the form $\Con_\alpha$, where $\alpha$ is a binumeration of some \compu\ axiomatisation 
of an extension $A$ of $\PA$ axiomatised by $\mc A$.
It is known that $\PA$ is effectively essentially $\mathbb{C}$-incomplete (cf.~Lindstr{\"o}m \cite[Theorem 2.8]{Lin03}). 
Thus, $\PA$ is effectively $\mathbb{C}(\PA)$-inseparable and effectively essentially $\mathbb{C}(\PA)$-incomplete. 

For each $n \geq 1$, let $\mathbb D_n$ be the function that assigns to sets of sentences $\mc A$ in the signature of arithmetic the set of all sentences $\phi$ such that $\phi$ and $\neg\, \phi$ are both $\Pi_n$-conservative over the theory axiomatised by $\mc A$. 
It is well-known that for any consistent \ce~extension $U$ of $\PA$, $\mathbb D_1(U) = \mathbb{O}(U)$ (cf.~\cite[Theorem 6.6]{Lin03}). 
For every pair of functions $\mc F$ and $\mc G$ from sets of sentences to sets of sentences, 
let $\mc F \cap \mc G$ be the function defined by $(\mc F \cap \mc G)(\mc A) = \mc F(\mc A) \cap \mc G(\mc A)$. 

For each binumeration $\alpha$ of an extension $U$ of $\PA$, the formula $\Prf_{\alpha}^{\Sigma_n}(x, y)$ is defined as
\[
    \exists u \leq y\, (\Sigma_n(u) \land {\sf true}_{\Sigma_n}(u) \land \Prf_\alpha(u \dot{\to} x, y)). 
\]
Then, we have the following: 

\begin{prop}[The small reflection principle (cf.~{\cite[Lemma 5.1(ii)]{Lin03}})]
Let $U$ be any \compu\ extension of $\PA$ and $\alpha$ be a binumeration of $U$. 
Then, for any sentence $\varphi$ and natural number $m$, we have \[U \vdash \exists y < \num{m}\, \Prf_{\alpha}^{\Sigma_n}(\gn{\varphi}, y) \to \varphi.\] 
\end{prop}

\begin{thm}
  For each $n \geq 1$, ${\sf PA}$ is effectively essentially $\mathbb{C} \cap \mathbb{D}_n$-incomplete.  
\end{thm}
\begin{proof}
Suppose that ${\sf W}_i$ is a consistent \ce~extension of $\PA$. 
By Craig's trick, we can effectively find a $k$ from $i$ such that ${\sf W}_k$ is a primitive \compu\ axiomatisation of ${\sf W}_i$. 
Let $\beta(x)$ be an effectively found primitive \compu\ binumeration of ${\sf W}_k$.
We can effectively find a formula $\alpha(x)$ such that: 
{\small
\[
    \PA \vdash \alpha(x) \leftrightarrow \Bigl( \bigl(\beta(x) \lor \exists y < x\, \Prf_{\beta}^{\Sigma_n}(\gn{\Con_\alpha}, y) \bigr) 
    \land \forall z < x\, \neg\, \Prf_{\beta}^{\Sigma_n}(\gn{\neg\, \Con_\alpha}, z) \Bigr).
\]
}
We show that $\Con_\alpha$ is $\Pi_n$-conservative over ${\sf W}_k$. 
Let $\pi$ be any $\Pi_n$ sentence such that ${\sf W}_k \cup \{\Con_\alpha\} \vdash \pi$. 
Then, ${\sf W}_k \cup \{\neg\, \pi\} \vdash \neg\, \Con_\alpha$. 
There exists a number $q$ such that $\PA \cup \{\neg\, \pi\} \vdash \Prf_{\beta}^{\Sigma_n}(\gn{\neg\, \Con_\alpha}, \num{q})$. 
Thus, $\PA \cup \{\neg\, \pi\} \vdash \alpha(x) \to x \leq \num{q}$. 
We have: {\small
\begin{align*}
	\PA \cup \{\neg\, \pi\}  \vdash \forall y < \num{q}\, \neg\, \Prf_{\beta}^{\Sigma_n}(\gn{\Con_\alpha}, y) & \to 
 \Bigl (\alpha(x) \to    \bigl( x\leq \num q \;\wedge \\
& \hspace*{1cm}
 \forall y < x\, \neg\, \Prf_{\beta}^{\Sigma_n}(\gn{\Con_\alpha}, y) \bigr)\Bigr) \\
 &
 \to \bigl(\alpha(x) \to (\beta(x) \land x \leq \num{q})\bigr) \\
& \to (\Con_{\beta \leq \num{q}} \to \Con_\alpha). \\
& \to \Con_\alpha
\end{align*}
}%
The last step uses the fact that $\beta(x)$ a is binumeration of ${\sf W}_k$ in combination with the the essential reflexiveness of $\PA$.
By combining this with the small reflection principle, we obtain
${\sf W}_k \cup \{\neg\, \pi\} \vdash \Con_\alpha$.
Since ${\sf W}_k \cup \{\Con_\alpha\} \vdash \pi$, we conclude ${\sf W}_k \vdash \pi$. 

We show that $\neg\, \Con_\alpha$ is also $\Pi_n$-conservative over ${\sf W}_k$.
Let $\pi$ be any $\Pi_n$ sentence such that ${\sf W}_k \cup \{\neg\, \Con_\alpha\} \vdash \pi$. 
Then, ${\sf W}_k \cup \{\neg\, \pi\} \vdash \Con_\alpha$. 
There exists a number $p$ such that $\PA \cup \{\neg\, \pi\} \vdash \Prf_{\beta}^{\Sigma_n}(\gn{\Con_\alpha}, \num{p})$. 
For each $q > p$, we have $\PA \cup \{\neg\, \pi\} \vdash \exists y < \num{q}\, \Prf_{\beta}^{\Sigma_n}(\gn{\Con_\alpha}, y)$. 
We obtain
\[
    \PA \cup \{\neg\, \pi\} \vdash \forall z < \num{q}\, \neg\, \Prf_{\beta}^{\Sigma_n}(\gn{\neg\, \Con_\alpha}, z) \to \alpha(\num{q}).
\]
So, for a sufficiently large $q > p$, we have 
\[
	\PA \cup \{\neg\, \pi\} \vdash \forall z < \num{q}\, \neg\, \Prf_{\beta}^{\Sigma_n}(\gn{\neg\, \Con_\alpha}, z) \to \neg\, \Con_\alpha.
\] 
By combining this with the small reflection principle, we have ${\sf W}_k \cup \{\neg\, \pi\} \vdash \neg\, \Con_\alpha$. 
Since ${\sf W}_k \cup \{\neg\, \Con_\alpha\} \vdash \pi$, we conclude ${\sf W}_k \vdash \pi$. 

We have proved that $\Con_\alpha \in \mathbb{D}_n({\sf W}_i)$. 
Consequently, we have that $\Con_{\alpha}$ is independent of ${\sf W}_i$. 

By $\Pi_n$-conservativity, we obtain that for each $m \in \omega$, 
\begin{itemize}
    \item ${\sf W}_k \vdash \forall y < \num{m}\, \neg\, \Prf_{\beta}^{\Sigma_n}(\gn{\Con_\alpha}, y)$ and
    \item ${\sf W}_k \vdash \forall z < \num{m}\, \neg\, \Prf_{\beta}^{\Sigma_n}(\gn{\neg\, \Con_\alpha}, z)$. 
\end{itemize}
Hence, we have ${\sf W}_k \vdash \alpha(\num{m}) \leftrightarrow \beta(\num{m})$. 
This means that $\alpha(x)$ is also a binumeration of ${\sf W}_k$, and, thus,
$\Con_{\alpha} \in \mathbb C({\sf W}_i)$. 
We take $\Phi(i) : = \Con_\alpha$. 
Then, $\Phi$ witnesses the effective essential $\mathbb{C} \cap \mathbb{D}_n$-incompleteness of $\PA$. 
\end{proof} 

\begin{cor}
$\PA$ is effectively $(\mathbb{C} \cap \mathbb{D}_n)(\PA)$-inseparable and effectively essentially $(\mathbb{C} \cap \mathbb{D}_n)(\PA)$-incomplete.
\end{cor}

\section{Effective ef-essential Incompleteness}\label{sec_ef}
Effective forms of incompleteness employ presentations of the extensions considered.
This is necessitated by the fact that our witnessing partial \compu\ functions need finite objects to operate on.
In the finite case,
we do have an alternative available to presenting an extension by a \ce~index. We can simply specify the sentence
we extend with. 
In this section, we study this notion of effective incompleteness for finite extensions and a variant. 
We consider the relation of these two notions to if-essential incompleteness.

\begin{defn}
Consider a \ce~theory $U$. We define:
\begin{itemize}
\item 
 $U$ is \textit{effectively ef-essentially $\mc F$-incomplete} iff 
 there exists a \compu\ function $\Phi$, such that, for any sentence $\phi$, if $\adj U\phi$ is consistent, 
 then $\Phi(\phi)$ is independent of $\adj U \phi$ and $\Phi(\phi) \in \mc F(\adj U\phi)$. 
    \item
$U$ is \textit{effectively ef-essentially $\mc F$-incomplete w.r.t.~$N$ and $\xi$} iff
$N$ is an interpretation of $\RR$ in $U$ and, for any sentence $\phi$, if $\adj U \phi$ is consistent, 
then $\xi(\gn{\phi})$ is independent of $\adj U \phi$ and $\xi(\gn{\phi}) \in \mc F(\adj U \phi)$.
Here the numerals are the numerals provided by $N$. To avoid heavy notation, we pretend that $N$ is one-dimensional. Of course, this is
inessential.
 \end{itemize}
\end{defn}

The notion of effective ef-essential incompleteness was studied by Jones~\cite{Jones69} as the name \textit{effective nonfinite completability}. 
We have the following simple insight.
\begin{thm}\label{if_to_ef}
Every effectively if-essentially $\mc F$-incomplete theory is effectively ef-essentially $\mc F$-incomplete.
\end{thm}

\begin{proof}
    This is immediate seeing that, given $\phi$, we can effectively find an index of $\adj U\phi$.
\end{proof}

However, an effectively ef-essentially incomplete theory can be decidable as we show in the following example.
So, the converse of Theorem~\ref{if_to_ef} does not generally hold. 

\begin{ex}\label{succ-example}
We consider the theory ${\sf Succ}^\circ$ in the language of zero and successor. The theory is axiomatised by: 
zero is not a successor; successor is injective; every number is either zero or a successor; for every $n$, there is at most one
successor-cycle of size $n$. 

One can show that ${\sf Succ}^\circ$ is decidable and that every sentence is equivalent to a Boolean combination of sentences ${\sf C}_n$ saying
`there is a cycle of size $n$'. See, e.g., \cite[Appendix A]{viss:nomi23}. We note that, over ${\sf Succ}^\circ$, the ${\sf C}_n$ are mutually independent.

Let $\phi$ be any ${\sf Succ}^\circ$-sentence. We can effectively find a sentence $\psi$, equivalent to $\phi$, which is a Boolean combination of the ${\sf C}_n$.
Let $k$ be the smallest number so that ${\sf C}_k$ does not occur in $\psi$. We set $\Phi(\phi) := {\sf C}_k$. It is easily seen that
$\Phi$ witnesses the effective ef-essential incompleteness of ${\sf Succ}^\circ$.
\end{ex}

We can say more about the difference of the two notions.
The following two theorems reveal an intrinsic difference between \emph{if-} and \emph{ef-}.

\begin{thm}\label{keukensmurf}
Suppose $\mc X$ is \ce~and $\Phi$ is a partial \compu\ witness that
$U$ is an effectively if-essentially $\mc X$-incomplete \ce~theory. 
Then, we can find, effectively from an index of $\Phi$, a $\phi \in \mc X$, such that $U \cup \{\phi\}$ is 
inconsistent. 
In particular, $U_{\mf p} \cup \mc X$ is not mono-consistent.
\end{thm}

\begin{proof} 
By the Recursion Theorem, we find an $i$ such that
\[
    {\sf W}_{i} = \begin{cases} U \cup \{\Phi(i)\} & \text{if}\ \Phi(i)\ \text{converges and}\ \Phi(i) \in \mc X, \\
    U & \text{otherwise}.
    \end{cases}
\]
If the second clause would obtain, ${\sf W}_i$ would be a consistent if-extension of $U$.
So, $\Phi(i)$ would converge to an element of $\mc X$. \emph{Quod non}.
Thus, only the first clause can be active. Hence, $\Phi(i)$ converges to an element $\phi$ of $\mc X$.
Then, ${\sf W}_i = U \cup \{\phi\}$ and ${\sf W}_i \vdash \Phi(i)$. 
By the effective if-essential $\mc X$-incompleteness of $U$, we have that $U \cup \{\phi\}$ is inconsistent. 
\end{proof}

\begin{thm}\label{tuinsmurf}
    Suppose $U$ is a consistent effectively ef-essentially incomplete \ce~theory. 
Then, there is a \ce~set $\mc X$ such that $U$ is effectively ef-essentially $\mc X$-incomplete
and every $\phi$ in $\mc X$ is consistent with $U$, i.o.w., $U_{\mf p} \cup \mc X$ is mono-consistent.
\end{thm}

\begin{proof}
    Suppose $U$ is effectively ef-essentially incomplete as witnessed by $\Phi$. 
    We take:
    \[
    \Psi(\phi) :=  \begin{cases} (\phi \to \Phi(\phi)) & \text{if}\ \Phi(\phi)\ \text{converges}, \\
    \text{undefined} & \text{otherwise}.
    \end{cases}
    \]
    Let $\mc X$ be the range of $\Psi$.
    
    If $\phi$ is inconsistent with $U$ and $\Phi(\phi)$ converges, then $U\vdash \phi \to \Phi(\phi)$ and, thus,
    $\Psi(\phi)$ is consistent with $U$. If $\phi$ is consistent with $U$, then $\Phi(\phi)$ converges
    and $\Phi(\phi)$ is independent of $U \cup \{\phi\}$ and, \emph{a fortiori},
    consistent with $U$. 
\end{proof}

Combining Theorems~\ref{keukensmurf} and~\ref{tuinsmurf}, we immediately find the following corollary:
\begin{cor}\label{non_ef_to_if}
    Suppose $U$ is a consistent effectively ef-essentially incomplete \ce~theory. 
Then, there is a \ce~set $\mc X$ such that $U$ is effectively ef-essentially $\mc X$-incomplete
but not effectively if-essentially $\mc X$-incomplete.
\end{cor}

The next theorem gives a condition under which effective ef-essential $\mc X$-incompleteness w.r.t.~some $N$ and $\xi$ implies effective essential $\mc X$-incompleteness. 

\begin{thm}\label{optimisticsmurf}
Let $U$ be a consistent \ce~theory and suppose $U$ is effectively ef-essentially 
$\mc X$-incomplete w.r.t.~$N$ and $\xi$. Then, $U$ is effectively essentially $\mc X$-incomplete via a witnessing function $\Phi$, 
where $\Phi(i)$ is of the form $\xi(\gn \sigma)$.

\end{thm}

\begin{proof}
We assume the conditions of the theorem.
Suppose that ${\sf W}_{i}$ is a consistent extension of $U$. 
We construct formulas $\gamma_0$ and $\gamma_1$ with some desired properties as follows.
Our construction is an adaptation of a solution to \cite[Exercise 3.4]{Lin03}.
\begin{enumerate}[{\sf St{a}ge 1.}]
    \item We start with $\Sigma_1$-formulas $\alpha_0$ and $\alpha_1$, where $\alpha_0$ represents ${\sf W}_{i\mf p}$
    and $\alpha_1$ represents ${\sf W}_{i\mf r}$. We arrange it so that the $\alpha_i$ start with a single existential
    quantifier followed by a $\Delta_0$-formula.
    \item 
    Let $\beta_0 := \alpha_0 \leq \alpha_1$ and $\beta_1 := \alpha_1 < \alpha_0$.
    We note that the $\beta_i$ represent the same sets as the $\alpha_i$. 
    Moreover, if $n \in {\sf W}_{i \mf p}$, then $\beta_0(\num n)$ is true, and, hence ${\sf R} \vdash \beta_{0}(\num n) \wedge \neg\, \beta_{1}(\num n)$. 
    Also, if $n \in {\sf W}_{i \mf r}$, then ${\sf R} \vdash \beta_{1}(\num n) \wedge \neg\, \beta_{0}(\num n)$. 
    \item 
    Let $\eta(x)$ be a $\Sigma_1$-formula with one existential quantifier followed by a $\Delta_0$-formula that
    is equivalent in predicate logic with $(\beta_0(x) \vee \beta_1(x))$.
    By the Fixed Point Lemma, we find a $\rho$, such that,
    for every $n$, we have ${\sf R} \vdash \rho(\num n) \iff \eta(\num n) \leq \alpha_0(\gn{\rho^N(\num n)}) $.
    By the reasoning of the FGH Theorem (see, e.g., \cite[Section 3]{viss:faith05} or~\cite{kura:FGHt23}), 
    we have: ${\sf W}_i \vdash \rho^N(\num n)$ iff $\rho(\num{n})$ is true iff $n \in {\sf W}_{i \mf p}
    \cup  {\sf W}_{i\mf r}$. We define $\gamma_i(x) := (\beta_i(x) \wedge \rho(x))$. We find:
    \begin{enumerate}[A.]
        \item The $\gamma_0$ and $\gamma_1$ represent ${\sf W}_{i \mf p}$ and ${\sf W}_{i \mf r}$, respectively.
        \item 
       If ${\sf W}_i \vdash (\gamma_0(\num n) \vee \gamma_1(\num n))^N$, then $n \in {\sf W}_{i\mf p}\cup {\sf W}_{i\mf r}$.
       \item If $n\in {\sf W}_{i\mf p}$, then ${\sf R} \vdash \gamma_0(\num n) \wedge \neg \, \gamma_{1}(\num n)$.
       \item
       If $n\in {\sf W}_{i\mf r}$, then ${\sf R} \vdash\gamma_1(\num n) \wedge \neg \, \gamma_{0}(\num n)$.
    \end{enumerate}
\end{enumerate}

We can effectively find a sentence $\sigma$ from $i$ satisfying the following equivalence:
\[
	U \vdash \sigma \leftrightarrow   \Bigl(\bigl(\gamma_0(\gn{\xi(\gn{\sigma})})^N \to \xi(\gn{\sigma}) \bigr)
     \land \bigl( \gamma_1(\gn{\xi(\gn{\sigma})})^N \to \neg\, \xi(\gn{\sigma})\bigr) \Bigr).
\]
Suppose, towards a contradiction, that $\xi(\gn{\sigma}) \in {\sf W}_{i \mf p} \cup {\sf W}_{i \mf r}$. 

\begin{itemize}
    \item Suppose $\xi(\gn{\sigma}) \in {\sf W}_{i \mf p}$. Then, by (C), we obtain $U \vdash \sigma \leftrightarrow \xi(\gn{\sigma})$.
Since $U \cup \{\sigma\} \vdash \xi(\gn{\sigma})$, we have that $U \cup \{\sigma\}$ is inconsistent because $U$ is ef-essentially $\mc X$-incomplete w.r.t.~$N$ and $\xi$. 
Thus, $U \vdash \neg\, \sigma$, and, hence, $U \vdash \neg\, \xi(\gn{\sigma})$. 
This contradicts the consistency of ${\sf W}_i$. 
    \item Suppose $\xi(\gn{\sigma}) \in {\sf W}_{i \mf r}$. Then, by (D), we obtain
    $U \vdash \sigma \leftrightarrow \neg\, \xi(\gn{\sigma})$.
Since $U \cup \{\sigma\} \vdash \neg\, \xi(\gn{\sigma})$, we have that $U \cup \{\sigma\}$ is inconsistent. 
Thus, $U \vdash \xi(\gn{\sigma})$. 
This contradicts the consistency of ${\sf W}_{i}$. 
\end{itemize}
We have shown that $\xi(\gn{\sigma})$ is independent of ${\sf W}_{i}$. 
By (B), we find that $U \nvdash (\gamma_0(\gn{\xi(\gn{\sigma})}) \lor \gamma_1(\gn{\xi(\gn{\sigma})}))^N$. 
Since \[U\vdash \neg\, \sigma \to \bigl (\gamma_0(\gn{\xi(\gn{\sigma})}) \lor \gamma_1(\gn{\xi(\gn{\sigma})})\bigr)^N,\]
we have that $U \cup \{\sigma\}$ is consistent. 
Hence, $\xi(\gn{\sigma}) \in \mc X$. 
We take $\Phi(i) : = \xi(\gn{\sigma})$. Thus,
$\Phi$ witnesses the effective essential $\mc X$-incompleteness of $U$. 
\end{proof}

We show that the converse of Theorem~\ref{optimisticsmurf} does not generally hold. 
For this, as compared with Corollary~\ref{cor_pour-el}, we prove the following proposition stating that the notion of effective ef-essential $\mc X$-incompleteness w.r.t.~$N$ and $\xi$ is not equivalent to the one obtained by replacing $\mc X$ with $[\mc X]_U$. 

\begin{prop}\label{ef_non_closure}
    If $U$ is effectively ef-essentially incomplete w.r.t.~$N$ and $\xi$, then there exists a \compu\ set $\mc X$ of formulas such that $U$ is effectively ef-essentially $[\mc X]_U$-incomplete w.r.t.~$N$ and $\xi$, but $U$ is not effectively ef-essentially $\mc X$-incomplete w.r.t.~$N$ and $\eta$ for all $\eta$. 
\end{prop}
\begin{proof}
    Suppose that $U$ is effectively ef-essentially incomplete w.r.t.~$N$ and $\xi$. 
    Let $\mc X$ be the \compu\ set defined by
    \[
        \mc X : = \{\xi(\gn{\phi}) \land \bigwedge_{i = 1}^{\gn{\phi}} (0=0)^N \mid \phi\ \text{is a}\ U\text{-sentence}\}. 
    \]
    Since each $\xi(\gn{\phi})$ is $U$-provably equivalent to some element of $\mc X$, we have that $U$ is effectively ef-essentially $[\mc X]_U$-incomplete w.r.t.~$N$ and $\xi$. 
    On the other hand, since there are unboundedly many $U$-sentences $\phi$ in the sense of G\"odel numbers such that $U \cup \{\phi\}$ is consistent, it is shown that there is no single formula $\eta$ such that $U$ is effectively ef-essentially $\mc X$-incomplete w.r.t.~$N$ and $\eta$. \end{proof}


\begin{rem}
    Suppose $U$ is a \ce~theory that interprets {\sf R} via $N$. Then, we can find a $\rho$ such that 
    $U$ is essentially ef-incomplete w.r.t. $N$ and $\rho$. This is a variant of Rosser's Theorem.
    It follows, by Theorem~\ref{optimisticsmurf}, that $U$ is effectively if-essentially incomplete.
    Moreover, this can be relativised
    to $\mc X$, for any $\mc X$ that contains all the $\rho(\gn \phi)$'s.
    From the results presented in this section, we have the following observations. 
    \begin{itemize}
        \item  By Theorem~\ref{keukensmurf}, we can effectively find a $\sigma$, such that
    $U \cup \{\rho(\sigma)\}$ is inconsistent. 
    
        \item By Corollary~\ref{non_ef_to_if}, we can find a \ce~set $\mc Y$, such that $U$ is
    effectively ef-essentially $\mc Y$-incomplete, but not effectively if-essentially $\mc Y$-incomplete.
    
        \item Finally, by Theorem~\ref{ef_non_closure}, we can find a \compu\ $\mc Z$ such that $U$ is effectively ef-essentially
    $[\mc Z]_U$-incomplete w.r.t.~$N$ and $\rho$, but not effectively ef-essentially
    $\mc Z$-incomplete w.r.t.~$N$ and $\eta$ for all $\eta$. 
    By Theorem~\ref{optimisticsmurf}, $U$ is effectively if-essentially $[\mc Z]_U$-incomplete, and so by Corollary~\ref{cor_pour-el}, we find that $U$ is effectively if-essentially $\mc Z$-incomplete.
    \end{itemize}
\end{rem}

We provide some versions of the converses of Theorems~\ref{if_to_ef} and~\ref{optimisticsmurf} 
when $U$ and $\mc X$ satisfy a certain condition.

\begin{thm}\label{truesmurf}
Let $U$ be a consistent \ce~theory and let $N$ be an interpretation of $\RR$ in $U$. 
Let $\mc X$ be a \ce~set of sentences. 
Suppose that we have a $U$-formula {\sf true} such that
 $U \vdash {\sf true}(\gn \phi) \iff \phi$, for all $\phi \in \mc X$.
Then, the following are equivalent:
\begin{enumerate}[a.]
    \item $U$ is effectively ef-essentially $[\mc X]_U$-incomplete w.r.t.~$N$ and $\xi$ for some $\xi$.
    \item $U$ is effectively if-essentially $\mc X$-incomplete.
    \item $U$ is effectively ef-essentially $\mc X$-incomplete.
\end{enumerate}   
\end{thm}
\begin{proof}
``(a) to (b)''. This follows from Theorem~\ref{optimisticsmurf} and Corollary~\ref{cor_pour-el}. 

\medskip

``(b) to (c)''. By Theorem~\ref{if_to_ef}. 

\medskip

``(c) to (a)''. 
    Let $\Phi$ witness the effective ef-essential $\mc X$-incompleteness of $U$. 
    By the representability theorem, we find a formula ${\sf G}_\Phi(x, y)$ representing $\Phi$ in $\RR$. 
    Define \[ \xi(x):=
        \exists y\in \delta_N \,  \bigl({\sf G}^N_\Phi(x, y) \land  {\sf true}(y) \bigr).\]
    We show that $U$ is effectively ef-essentially $[\mc X]_U$-incomplete w.r.t.~$N$ and $\xi$. 
    Suppose that $U \cup \{\phi\}$ is consistent and, thus, $\Phi(\varphi) \in \mc X$ and $\Phi(\phi)$ is independent of $\adj U\phi$. 
    We get:  
    \begin{eqnarray*}
        U \vdash \xi(\gn{\phi}) & \iff & \exists y\in \delta_N \, \bigl({\sf G}^N_\Phi(\gn{\phi}, y)  \land {\sf true}(y) \bigr) \\
        & \iff &{\sf true}(\gn{\Phi(\phi)}) \\
        & \iff& \Phi(\phi). 
    \end{eqnarray*}
    Thus, $\xi(\gn{\phi})$ is also independent of $U \cup \{\phi\}$ and we find $\xi(\gn{\phi}) \in [\mc X]_U$. 
\end{proof}

\begin{rem}
    One might wonder whether there is an infinitary version for the ef-notions. It seems that this will be
    non-trivial to attain. The reason is that we can use Craig's trick to transform every index $i$ effectively
    to an index $i^\ast$, so that $(U+{\sf W}_i)_{\mf p} = (U+{\sf W}_{i^\ast})_{\mf p}$ and ${\sf W}_{i^\ast}$ is
    \compu\. We can even make ${\sf W}_{i^\ast}$ p-time decidable. In case $U$ is a pair theory, we can even
    make ${\sf W}_{i^\ast}$ a scheme. See~\cite{vaug:axio67} or~\cite{viss:vaug12}. So, we would have to look for some really different
    notion of extension.
\end{rem}

\section{Heredity}\label{heredity}
In this section, we study what happens when we consider effective notions combined with hereditariness. Our main insight here is that
an adapted version of Pour-El's result also holds in this case.

\begin{defn}
A consistent \ce~theory $U$ is \textit{effectively essentially hereditarily $\mc X$-creative} iff there exists a partial \compu\ function $\Phi$ such that for any $i, j, k \in \omega$, if ${\sf W}_k$ is a consistent extension of $U$, ${\sf W}_{i}$ is a subtheory of ${\sf W}_k$, and ${\sf W}_{i \mf p} \cap {\sf W}_{j} = \emptyset$, then $\Phi(i, j, k)$ converges, $\Phi(i, j, k) \in \mc X$, and $\Phi(i, j, k) \notin {\sf W}_{i \mf p} \cup {\sf W}_j$. 
\end{defn}

\begin{lem}\label{eehc1}
For any consistent \ce~theory $U$, the following are equivalent: 
\begin{enumerate}[i.]
	\item $U$ is effectively essentially hereditarily $\mc X$-creative. 
	\item There exists a partial \compu\ function $\Psi$ such that for any $i, j \in \omega$, if ${\sf W}_{i}$ is a theory consistent with $U$ and 
 ${\sf W}_{i \mf p} \cap {\sf W}_{j} = \emptyset$, then $\Psi(i, j)$ converges, $\Psi(i, j) \in \mc X$, and $\Psi(i, j) \notin {\sf W}_{i \mf p} \cup {\sf W}_{j}$. 
\end{enumerate}
\end{lem}
\begin{proof}
``(i) to (ii)''. Let $\Phi(i, j, k)$ be a partial \compu\ function witnessing the effective essential hereditary creativity of $U$. 
We define a partial \compu\ function $\Psi$ by $\Psi(i, j) : = \Phi(i, j, k)$, where $k$ is an effectively found index of $(U \cup {\sf W}_{i})_{\mf p}$. 
It is easy to see that $\Psi$ witnesses condition (ii). 

``(ii) to (i)''. Let $\Psi$ be a partial \compu\ function that witnesses condition (ii). 
We define a partial \compu\ function $\Phi$ by $\Phi(i, j, k) : = \Psi(i, j)$ for all $k$. 
It is easy to see that $\Phi$ witnesses the effective essential hereditary $\mc X$-creativity of $U$. 
\end{proof}

\begin{thm}\label{eehc2}
For any consistent \ce~theory $U$ and set $\mc X$ of sentences, the following are equivalent: 
\begin{enumerate}[a.]
	 \item $U$ is strongly effectively $\mc X$-inseparable.
\item $U$ is effectively essentially hereditarily $\mc X$-creative. 
\end{enumerate}
\end{thm}
\begin{proof}
``(a) to (b)''. Let $\Phi$ be a partial \compu\ function that witnesses the strong effective $\mc X$-inseparability of $U$. 
Suppose that $U \cup {\sf W}_i$ is consistent and ${\sf W}_{i \mf p} \cap {\sf W}_j = \emptyset$. 
we can effectively find numbers $k_0$ and $k_1$ from $i$ and $j$ such that ${\sf W}_{k_0} = {\sf W}_{i \mf p}$ and ${\sf W}_{k_1} = U_{\mf r} \cup {\sf W}_j$. 
Clearly, ${\sf W}_{k_0}$ and ${\sf W}_{k_1}$ biseparate $0_{U \mf p}$ and $U_{\mf r}$. 
By strong effective $\mc X$-inseparability,  $\Phi(k_0, k_1)$ converges, $\Phi(k_0, k_1) \in \mc X$, and $\Phi(k_0, k_1) \notin {\sf W}_{k_0} \cup {\sf W}_{k_1}$. 
We take $\Psi(i, j) : = \Phi(k_0, k_1)$. 
Then, $\Psi$ satisfies Condition (ii) of Lemma~\ref{eehc1}, and, hence, $U$ is effectively essentially hereditarily $\mc X$-creative. 

``(b) to (a)''. Suppose that $U$ is effectively essentially hereditarily $\mc X$-creative. 
Let $\Psi$ be a partial \compu\ function that witnesses Condition (ii) of Lemma~\ref{eehc1}. 
Suppose that ${\sf W}_{i}$ and ${\sf W}_{j}$ 
biseparate $0_{U \mf p}$ and $U_{\mf r}$. 
By the Double Recursion Theorem with parameters (cf.~\cite[Exercise 3.15.(b)]{Soar87}), we can effectively find numbers $k_0$ and $k_1$ from $i$ and $j$ such that
\begin{itemize}
	\item ${\sf W}_{k_0} = \begin{cases} \{\phi\} & \text{if}\ \Psi(k_0, k_1) \simeq \phi\ \text{and}\ 
\phi \in {\sf W}_{i}, \\
	\emptyset & \text{otherwise.}\end{cases}$

	\item ${\sf W}_{k_1} = \begin{cases} \{\neg\, \phi\}_{\mf r} & \text{if}\ \Psi(k_0, k_1) \simeq \phi\ \text{and}\ 
 \phi \in {\sf W}_{j},\\
	0_{U \mf r} & \text{otherwise.}\end{cases}$
\end{itemize}

Suppose, towards a contradiction, that $\Psi(k_0, k_1) \simeq \phi$ and $\phi \in {\sf W}_{i} \cup {\sf W}_{j}$. 
\begin{itemize}
    \item  If 
    $\phi \in {\sf W}_{i}$,
    then ${\sf W}_{k_0} = \{\phi\}$ and ${\sf W}_{k_1} = 0_{U \mf r}$. 
    Since ${\sf W}_i \cap U_{\mf r} = \emptyset$, we have $\phi \notin U_{\mf r}$, and hence $U \cup {\sf W}_{k_0}$ is consistent. 
    If $\xi \in {\sf W}_{k_0 \mf p} \cap {\sf W}_{k_1}$ for some $\xi$, then $\phi \vdash \xi$ and $0_U \vdash \neg\, \xi$. 
    This contradicts $\phi \notin U_{\mf r}$. 
    Hence, ${\sf W}_{k_0 \mf p} \cap {\sf W}_{k_1} = \emptyset$, so we obtain $\phi = \Psi(k_0, k_1) \notin {\sf W}_{k_0 \mf p} \cup {\sf W}_{k_1}$. 
    This contradicts $\phi \in {\sf W}_{k_0 \mf p}$. 

    \item If 
    $\phi \in {\sf W}_{j}$, then ${\sf W}_{k_0} = \emptyset$ and ${\sf W}_{k_1} = \{\neg\, \phi\}_{\mf r}$. 
    If $\xi \in {\sf W}_{k_0 \mf p} \cap {\sf W}_{k_1}$ for some $\xi$, then $0_U \vdash \xi$ and $\neg\, \phi \vdash \neg\, \xi$. 
    We obtain $\phi \in 0_{U\mf p}$, and this contradicts ${\sf W}_j \cap 0_{U\mf p} = \emptyset$. 
    Thus, we have ${\sf W}_{k_0 \mf p} \cap {\sf W}_{k_1} = \emptyset$. 
    Since $U \cup {\sf W}_{k_0} = U$ is consistent, we obtain $\phi \notin {\sf W}_{k_0 \mf p} \cup {\sf W}_{k_1}$. 
    This contradicts $\phi \in {\sf W}_{k_1}$. 
\end{itemize}

    We have shown ${\sf W}_{k_0} = \emptyset$ and ${\sf W}_{k_1} = 0_{U \mf r}$. 
    Since $U \cup {\sf W}_{k_0}$ is consistent and ${\sf W}_{k_0 \mf p} \cap {\sf W}_{k_1} = \emptyset$, 
    we obtain that $\Psi(k_0, k_1)$ converges and $\Psi(k_0, k_1) \in \mc X$. 
    We have also shown $\Psi(k_0, k_1) \notin {\sf W}_i \cup {\sf W}_j$. 
    We take $\Phi(i, j) : = \Psi(k_0, k_1)$. Thus, $\Phi$ witnesses the strong effective $\mc X$-inseparability of $U$. 
\end{proof}

We note that the proof in the (b) to (a) direction only employs finitely axiomatised theories.

We also study an adapted version of the result established in Subsection \ref{doublesmurf}. 
We say that a theory $U$ is \textit{strongly doubly $\mc X$-generative} iff there exists a total \compu\ function $\Phi$ such that for any $i, j \in \omega$, if ${\sf W}_i \cap {\sf W}_j = \emptyset$, then 
\begin{itemize}
    \item $\Phi(i, j) \in 0_{U \mf p}$ iff $\Phi(i, j) \in {\sf W}_j$,
    \item $\Phi(i, j) \in U_{\mf r}$ iff $\Phi(i, j) \in {\sf W}_i$, 
    \item if $\Phi(i, j) \notin {\sf W}_i \cup {\sf W}_j$, then $\Phi(i, j) \in \mc X$. 
\end{itemize}

\begin{thm}\label{eehc3}
For any consistent \ce~theory $U$, the following are equivalent: 
\begin{enumerate}[a.]
    \item $U$ is strongly doubly $\mc X$-generative. 
 \item $U$ is effectively essentially hereditarily $\mc X$-creative.
\end{enumerate}
\end{thm}
\begin{proof}
``(a) to (b)''. Suppose that a total \compu\ function $\Phi$ witnesses the strong double $\mc X$-generativity of $U$. 
Suppose that ${\sf W}_i$ is a theory consistent with $U$ and ${\sf W}_{i \mf p} \cap {\sf W}_j = \emptyset$. 
We can effectively find $k$ from $i$ such that ${\sf W}_{k} ={\sf W}_{i \mf p}$. 
Since ${\sf W}_{k} \cap {\sf W}_{j} = \emptyset$, we have that $\Phi(k, j) \in 0_{U \mf p}$ iff $\Phi(k, j) \in {\sf W}_{j}$; and $\Phi(k, j) \in U_{\mf r}$ iff $\Phi(k, j) \in {\sf W}_{i \mf p}$. 
Since $0_{U \mf p} \cap {\sf W}_{j} = U_{\mf r} \cap {\sf W}_{i \mf p} = \emptyset$, we obtain $\Phi(k, j) \notin {\sf W}_{i \mf p} \cup {\sf W}_{j}$ and $\Phi(k, j) \in \mc X$. 
We take $\Psi(i, j) : = \Phi(k, j)$. Thus, $\Psi$ satisfies Condition (ii) of Lemma~\ref{eehc1}. 
So, $U$ is effectively essentially hereditarily $\mc X$-creative. 

``(b) to (a)''. Suppose that $U$ is effectively essentially hereditarily $\mc X$-creative. 
Let $\Psi$ be a partial \compu\ function that witnesses Condition (ii) of Lemma~\ref{eehc1}. 
We may assume that $\Psi$ is a total function. 
By the Recursion Theorem with parameters, there exist total \compu\ functions $\Theta_0$ and $\Theta_1$ such that, setting
 $\Phi^*(x, y) : = \Psi(\Theta_0(x, y), \Theta_1(x, y))$, we have:
{\small
\[
	{\sf W}_{\Theta_0(x, y)} = \begin{cases} 
    \{\Phi^*(x, y)\} & \ \text{if}\ \bigl(\Phi^*(x, y) \in (0_{U \mf p} \cup {\sf W}_{x}) \bigr) \leq \bigl(\Phi^*(x, y) \in (U_{\mf r} \cup {\sf W}_{y}) \bigr),\\
	\emptyset & \ \text{otherwise}. \end{cases}
\]
\[
	{\sf W}_{\Theta_1(x, y)} = \begin{cases} \{\neg\, \Phi^*(x, y)\}_{\mf r}  &
 \ \text{if}\ \bigl(\Phi^*(x, y) \in (U_{\mf r} \cup {\sf W}_{y}) \bigr) < \bigl(\Phi^*(x, y) \in (0_{U \mf p} \cup {\sf W}_{x}) \bigr), \\
	0_{U \mf r} & \ \text{otherwise}. \end{cases}
\]
}%
We show that $\Phi^*$ witnesses the strong double $\mc X$-generativity of $U$.
Let $i$ and $j$ be such that ${\sf W}_i \cap {\sf W}_j = \emptyset$. We set $\phi^\ast :=\Phi^*(i, j)$.
\newcommand{\phia}{\phi^\ast}

\begin{itemize}
	\item Suppose $\bigl(\phia \in (0_{U \mf p} \cup {\sf W}_{i}) \bigr) \leq \bigl(\phia \in (U_{\mf r} \cup {\sf W}_{j}) \bigr)$ holds. 
    Then, ${\sf W}_{\Theta_0(i, j)} = \{\phia\}$ and ${\sf W}_{\Theta_1(i, j)} = 0_{U \mf r}$. 
 Since, \[\Psi(\Theta_0(i, j), \Theta_1(i, j)) = \phia \in 
 {\sf W}_{\Theta_0(i, j) \mf p} \subseteq {\sf W}_{\Theta_0(i, j) \mf p} \cup {\sf W}_{\Theta_1(i, j)},\] 
 by  Condition (ii) of Lemma~\ref{eehc1}, we have that (I) ${\sf W}_{\Theta_0(i, j)}$ is inconsistent with $U$ or 
 (II) ${\sf W}_{\Theta_0(i, j) \mf p} \cap {\sf W}_{\Theta_1(i, j)} \neq \emptyset$. 
 If (I) holds, then $\phia \in U_{\mf r}$. 
 If (II) holds, then $\phia \vdash \xi$ and $0_U \vdash \neg\, \xi$, for some $\xi$, and hence $\phia \in 0_{U \mf r} \subseteq U_{\mf r}$. 
 Therefore, in either case,  $\phia \in U_{\mf r}$. 
 Since, $0_{U \mf p} \cap U_{\mf r} = {\sf W}_i \cap {\sf W}_j = \emptyset$ and 
 $\phia \in {\sf W}_i \cup 0_{U \mf p}$, we obtain 
 $\phia \notin 0_{U \mf p} \cup {\sf W}_j$ and $\phia \in U_{\mf r} \cap {\sf W}_i$. 
 
 	\item Suppose $\bigl(\phia \in (U_{\mf r} \cup {\sf W}_{j}) \bigr) < \bigl(\phia \in (0_{U \mf p} \cup {\sf W}_{i}) \bigr)$ holds.
    Then, ${\sf W}_{\Theta_0(i, j)} = \emptyset$ and ${\sf W}_{\Theta_1(i, j)} = \{\neg\, \phia \}_{\mf r}$. 
 Since ${\sf W}_{\Theta_0(i, j)}$ is consistent with $U$ and 
 \[\Psi(\Theta_0(i, j), \Theta_1(i, j)) = \phia \in {\sf W}_{\Theta_1(i, j)} 
 \subseteq {\sf W}_{\Theta_0(i, j) \mf p} \cup {\sf W}_{\Theta_1(i, j)},\] 
 by Condition (ii) of Lemma~\ref{eehc1}, 
 we have that $\xi \in {\sf W}_{\Theta_0(i, j) \mf p} \cap {\sf W}_{\Theta_1(i, j)}$ for some $\xi$. 
Then $0_U \vdash \xi$ and $\neg\, \phia \vdash \neg\, \xi$, and hence $\phia \in 0_{U \mf p}$. 
 Since, $0_{U \mf p} \cap U_{\mf r} = {\sf W}_i \cap {\sf W}_j = \emptyset$ and $\phia \in  U_{\mf r}\cup {\sf W}_j$, 
 we obtain $\phia \in 0_{U \mf p} \cap {\sf W}_j$ and $\phia \notin U_{\mf r} \cup {\sf W}_i$. 

    \item Otherwise, $\phia \notin 0_{U \mf p} \cup {\sf W}_j$ and $\phia \notin U_{\mf r} \cup {\sf W}_i$.
    In this case, ${\sf W}_{\Theta_0(i, j)} = \emptyset$ and ${\sf W}_{\Theta_1(i, j)} = 0_{U \mf r}$. 
    Since ${\sf W}_{\Theta_0(i, j)}$ is consistent with $U$ and 
    ${\sf W}_{\Theta_0(i, j) \mf p} \cap {\sf W}_{\Theta_1(i, j)} = \emptyset$, 
    we obtain $\phia = \Psi(\Theta_0(i, j), \Theta_1(i, j)) \in \mc X$. 
\end{itemize}
Thus, $\Phi^*$ witnesses the strong double $\mc X$-generativity of $U$. 
\end{proof}

We prove a closure property for strong effective inseparability.
For a theory $U$ and formula classes $\mc X$ and $\mc Y$, we write:
\begin{itemize}
\item
$\mc X \wedge \mc Y$ for $\verz{(\phi \wedge \psi)\mid \phi \in \mc X\text{ and } \psi \in \mc Y}$. 
\item
$U \domi \mc X$ iff, for all $\phi\in \mc X$, we have $\phi \vdash U$.
\end{itemize}

\begin{thm}\label{domi}
Suppose $ U \domi  {\mc X}$. 
Suppose further that $U$ is $\mc X$-creative and effectively $\mc Y$-inseparable. 
Then, $U$ is strongly effectively $\mc X\wedge \mc Y$-inseparable.
\end{thm}

\begin{proof}
Suppose $U \domi  \mc X$. Let $U$ be $\mc X$-creative and effectively $\mc Y$-inseparable. 
Suppose ${\sf W}_i$ and ${\sf W}_j$ bi-separate $U_{\mf p}$ and $0_{U\mf r}$.

Let $\mc W :=  \verz{\psi \mid \exists \phi \in {\sf W}_j\; U \cup \{\psi\} \vdash \phi}$. 
Suppose $\psi \in U_{\mf p}\cap\mc W$. Then, for some $\phi \in {\sf W}_j$, we have $U \cup \{\psi\} \vdash  \phi$ and $U \vdash \psi$. So, $U \vdash \phi$. \emph{Quod non}.
By $\mc X$-creativity,
we can effectively find a $\phi^\ast\in \mc X$ such that $\phi^\ast\not\in U_{\mf p} \cup \mc W$.     
We claim that (\dag) $\verz{\phi^\ast}_{\mf p} \cap {\sf W}_j = \emptyset$. If not, $\phi^\ast \vdash \chi$, for some $\chi\in {\sf W}_j$.
So, \emph{a fortiori},  $\phi^\ast\in\mc W$. \emph{Quod non}.

Let ${\sf W}_{i^\ast} := \verz{\psi \mid (\phi^\ast \wedge \psi) \in {\sf W}_i}\cup \verz{\phi^\ast}_{\mf p}$ and 
${\sf W}_{j^\ast} := \verz{\psi \mid (\phi^\ast \wedge \psi) \in {\sf W}_j}$.
Suppose $\psi \in {\sf W}_{i^\ast} \cap {\sf W}_{j^\ast}$. Then $\phi^\ast \vdash \psi$ and $(\phi^\ast \wedge \psi) \in {\sf W}_j$.
It follows that $(\phi^\ast\wedge \psi) \in  \verz{\phi^\ast}_{\mf p} \cap {\sf W}_j $. \emph{Quod non}, by (\dag).
Since $\phi^\ast\in \mc X$, we find that $U_{\mf p} \subseteq \verz{\phi^\ast}_{\mf p} \subseteq {\sf W}_{i^\ast}$.
We have:
\begin{eqnarray*}
\phi \in U_{\mf r} & \To & \phi^\ast \vdash \neg\,\phi \\
& \To & 0_U \vdash \neg \,(\phi^\ast \wedge \phi) \\
& \To & (\phi^\ast \wedge \phi)\in 0_{U\mf r} \\
& \To & (\phi^\ast \wedge \phi)\in {\sf W}_j \\
& \To & \phi\in {\sf W}_{j^\ast}
\end{eqnarray*}
So $U_{\mf r} \subseteq {\sf W}_{j^\ast}$.

We can effectively find a $\psi^\ast\in \mc Y$, such that $\psi^\ast \not\in {\sf W}_{i^\ast}\cup {\sf W}_{j^\ast}$.  So, $(\phi^\ast \wedge\psi^\ast) \not\in {\sf W}_i \cup{\sf W}_j$
and $(\phi^\ast \wedge\psi^\ast)\in \mc X \wedge \mc Y$. 
\end{proof}


\begin{ex}
The strong effective inseparability of the theory $\RR$ was proved by Vaught \cite[5.2]{Vau62}. 
In \cite[Theorem 2.4]{KV23}, the authors provided new proofs of this fact. 
Here, we give a proof of this fact in terms of Theorem~\ref{domi}, based on the method developed in~\cite{KV23}.
In \cite[Generalised Certified Extension Theorem]{KV23}, it is proved that for any $\Sigma_1$-sentence
$\sigma$, we can effectively find a sentence $[\sigma]$ satisfying the following conditions:
\begin{enumerate}[i.]
    \item $[\sigma] \vdash \sigma$, 
    \item if $\sigma$ is true, then $\RR \vdash [\sigma]$, 
    \item if $\sigma$ is false, then $[\sigma] \vdash \RR$. 
\end{enumerate}
Let $\mc X = \{\phi \mid \phi \vdash \RR\}$, then $\RR \domi \mc X$. 
Since $\RR$ is not finitely axiomatisable, we have $\RR_{\mf p} \cap \mc X = \emptyset$. 
We show that $\RR$ is $\mc X$-creative. 
Let ${\sf W}_i$ be any \ce~set disjoint from $\RR_{\mf p}$. 
From $i$, we effectively find a $\Sigma_1$ sentence $\rho$ satisfying
\[
    \RR \vdash \rho \leftrightarrow [\rho] \in {\sf W}_i.
\]
If $\rho$ were true, then $[\rho] \in {\sf W}_i$. 
Also by (ii), $[\rho] \in \RR_{\mf p}$, a contradiction. 
Thus, $\rho$ is false, which implies $[\rho] \notin {\sf W}_i$. 
By (iii), we get $[\rho] \vdash \RR$, and so we find $[\rho] \in \mc X$. 
Since $\RR_{\mf p} \cap \mc X = \emptyset$, we also have $[\rho] \notin \RR_{\mf p}$. 
Then, the partial \compu\ function $\Phi$ defined by $\Phi(i) : = [\sigma]$ witnesses the $\mc X$-creativity of $\RR$. 

Since $\RR$ is effectively ${\sf sent}$-inseparable, by Theorem~\ref{domi}, we have that $\RR$ is strongly effectively $(\mc X \land {\sf sent})$-inseparable. 
%
\end{ex}

It is known that there exists a consistent \ce~theory which is effectively inseparable but is not essentially hereditarily undecidable (\cite[Example 6]{Vis22}). 
Relating to this example, we propose the following problem. 

 \begin{prob}
     For any effectively inseparable consistent \ce~theory $U$, can we find a formula class $\mc X$ such that
     $U$ is effectively $\mc X$-inseparable but not strongly effectively $\mc X$-inseparable?
 \end{prob}

The relationships between effective notions we have considered so far are visualised in Figure~\ref{Fig2}. 

\begin{figure}[ht]
\centering
\begin{tikzpicture}
{\scriptsize
\node (EEC1) at (0,0.3) {effectively};
\node (EEC2) at (0,0) {essentially $\mc X$-creative};

\node (EEI1) at (4,0.3) {effectively};
\node (EEI2) at (4,0) {essentially $\mc X$-incomplete};

\node (EefEI1) at (4,-0.7) {effectively};
\node (EefEI2) at (4,-1) {ef-essentially $\mc X$-incomplete};

\node (EI) at (0,2) {effectively $\mc X$-inseparable};

\node (EUEI1) at (4,2.3) {effectively uniformly};
\node (EUEI2) at (4,2) {essentially $\mc X$-incomplete};

\node (EEHC1) at (-4,1.3) {effectively essentially};
\node (EEHC2) at (-4, 1) {hereditarily $\mc X$-creative};

\node (SEI1) at (-4,3.3) {strongly};
\node (SEI2) at (-4,3) {effectively $\mc X$-inseparable};

\draw [<->, double] (EEC2)--(EEI2);
\draw [<->, double] (EI)--(EUEI2);
\draw [<->, double] (EI)--(EEC1);
\draw [<->, double] (EUEI2)--(EEI1);
\draw [->, double] (EEHC2)--(EEC1);
\draw [<->, double] (SEI2)--(EEHC1);
\draw [->, double] (SEI2)--(EI);
\draw [->, double] (EEI2)--(EefEI1);
}
\end{tikzpicture}
\caption{Implications between effective notions}\label{Fig2}
\end{figure}
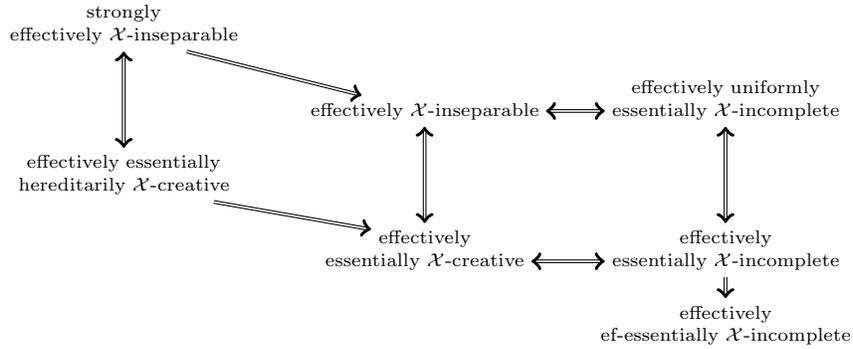

\section*{Acknowledgments}
We are grateful to Yong Cheng for his helpful comments.

\bibliographystyle{alpha}
\bibliography{reference2}

\appendix

\section{Pour-El's Original Proof}\label{original}

For the sake of completeness, we reproduce Pour-El's original argument here. Where our argument in Section~\ref{mainsmurf}, is a direct diagonal argument,
Pour-El's argument proceeds by embedding disjoint pairs of \ce~sets in the given theory. 

We say that
$\tupel{\mc X,\mc Y}\leq_{\mf s}\tupel{\mc Z,\mc W}$, or $\tupel{\mc X,\mc Y}$ is \emph{semi-reducible} to 
$\tupel{\mc Z,\mc W}$, iff there is a \compu\ $\Phi$ such that, for all $n$, if
$n\in \mc X$, then $\Phi(n) \in \mc Z$, and, if $n\in \mc Y$, then $\Phi(n) \in \mc W$.
(As far as we can trace it, this notion is due to Smullyan (cf.~\cite{Smullyan60,Smullyan93,Cheng23}.)

\begin{lem}\label{ei}
Suppose $U$ is effectively if-essentially incomplete as witnessed by a partial \compu\ function $\Phi$. 
Let $\mc X$ and $\mc Y$ be disjoint \ce~sets. 
Then $\tupel{\mc X, \mc Y} \leq_{\mf s} \tupel{U_{\mf p},U_{\mf r}}$.
An index of a witness for the semi-reducibility can be effectively obtained from an index of $\Phi$ and the indices of $\mathcal X$, and $\mathcal Y$.
\end{lem}

The argument is an adaptation of Smullyan's argument that every disjoint pair of \ce~sets is 1-reducible to any
given effectively inseparable pair. See \cite[Exercise 11.29]{roge:theo67} or \cite[Exercise 2.4.18]{soar:turi16}.
We work with semi-reducibility rather than one-one reducibility to keep the argument as simple as possible.

\begin{proof}
We may assume that $\Phi(i)$ converges if ${\sf W}_i$ is an inconsistent theory. 
By the Recursion Theorem with a parameter, there exists a total \compu\ function $\Psi(x)$ such that, setting $\Phi^\ast(x) := \Phi(\Psi(x))$, 

\[{\sf W}_{\Psi(x)} = \begin{cases} U \cup \{\Phi^\ast(x)\} & \text{if}\ \Phi^\ast(x)\ \text{converges and}\ x \in \mc Y,
\\
	U \cup \{ \neg\, \Phi^\ast(x)\} & \text{if}\ \Phi^\ast(x)\ \text{converges and}\ x \in \mc X, \\
 U & \text{otherwise.}\end{cases}\]

\noindent
 We show that $\Phi^\ast$ witnesses the semi-reducibility of $\tupel{\mc X,\mc Y}$ to $\tupel{U_{\mf p},U_{\mf r}}$.


For any $n$, ${\sf W}_{\Psi(n)}$ is either a finite consistent extension of $U$ or an inconsistent theory, and so $\Phi^\ast(n) = \Phi(\Psi(n))$ converges. 
Thus, $\Phi^\ast$ is a total \compu\ function. 

Suppose $n\in \mc X$. Then, 
${\sf W}_{\Psi(n)}= U \cup \{\,\neg\, \Phi^\ast(n)\}$. 
Since, ${\sf W}_{\Psi(n)}\vdash \neg\,\Phi(\Psi(n))$,
we find, by our assumption on $\Phi$, 
that ${\sf W}_{\Psi(n)}$ is inconsistent. It follows that $U \vdash \Phi^*(n)$. 


The other case is similar.
%
\end{proof}

\begin{lem}[Smullyan {\cite[Proposition 1]{Smullyan60}}]\label{uppresei}
Suppose $\tupel{\mc X,\mc Y}$ is effectively inseparable and 
$\tupel{\mc X,\mathcal Y}\leq_{\mf s}\tupel{\mc Z,\mc W}$.
Then, $\tupel{\mc Z,\mc W}$ is effectively inseparable. 
\end{lem}

\begin{proof}
Suppose $\Theta$ witnesses the effective inseparability of 
$\mc X$ and $\mc Y$. Suppose further that $ \Psi$ witnesses the
semi-reducibility of $\tupel{\mc X,\mc Y}$ to $\tupel{\mc Z,\mc W}$.

Suppose the pair $\tupel{\mc Z',\mc W'}$ 
separates $\tupel{\mc Z,\mc W}$.
Say, the indices of $\mc Z'$ and $\mc W'$ are $i$ and $j$.
Let $\mc X' := \verz{n\mid \Psi(n) \in \mc Z'}$ and 
$\mc Y' := \verz{m\mid  \Psi(m) \in \mc W'}$.
Clearly, $\tupel{\mc X',\mc Y'}$ separates 
$\tupel{\mc X,\mc Y}$.

We can effectively find indices $k$ and $\ell$ for $\mc X'$ and 
$\mc Y'$ from $i$ and $j$.
Let $s:=  \Psi(\Theta(k,\ell))$. Suppose $s\in \mc Z'$. 
In that case   $\Theta(k,\ell)$ is in $\mc X'$. \emph{Quod non}.
Similarly, $s \not \in \mc W'$.
\end{proof}

\begin{thm}
Suppose $U$ is effectively if-essentially incomplete. 
Then, $U$ is effectively inseparable.
\end{thm}

\begin{proof}
Suppose $U$ is essentially if-effectively incomplete. Let $\tupel{\mc X,\mc Y}$ 
be any effectively inseparable pair of sets. We find, by Lemma~\ref{ei}, that
$\tupel{\mathcal X,\mathcal Y}\leq_{\mf s}\tupel{U_{\mf p},U_{\mf r}}$. So, by Lemma~\ref{uppresei}, also $\tupel{U_{\mf p},U_{\mf r}}$
will be effectively inseparable.
\end{proof}

Inspecting our argument, we see that, if $\Phi$ is a witness of the effective essential 
incompleteness, then
the witnesses of essential inseparability are in the range of $\Phi$.

\begin{rem}
    The function $\Phi^\ast$ of the proof of Lemma~\ref{ei} is actually a witness of many-one reducibility.
    We can see the other direction as follows. 
    
    Let $\Phi^\ast$ be as in the proof of Lemma~\ref{ei}.
    If $U \vdash \Phi^\ast(n)$, then, \emph{a fortiori}, 
${\sf W}_{\Psi(n)}\vdash \Phi(\Psi(n))$ and, thus, that ${\sf W}_{\Psi(n)}$ is inconsistent. 
It follows that $n\in \mc X$ or $n\in \mc Y$, since otherwise
${\sf W}_{\Psi(n)}$ would be the supposedly consistent theory $U$. 
But if $n\in \mc Y$, we find that $U \vdash \neg\, \Phi^\ast(n)$, \emph{quod non}, since  we assumed that 
$U\vdash \Phi^\ast(n)$ and that $U$ is consistent. So $n\in \mc X $. The other case is similar.

By padding we can assure that $\Phi^\ast$ is injective, thus witnessing one-one reducibility. 
However, this construction will not preserve the fact that the range
of $\Phi^\ast$ is contained in the range of $\Phi$. (It still will be so modulo provable equivalence in predicate logic.)
\end{rem}

\section{Effective Local Interpretability?}\label{locosmurf}

All formulations of \emph{what effective local interpretability is} that we could think of either collapse to ordinary local interpretability or to
interpretability. Let $\mathfrak{id}_V$ denote the conjunction of the equality axioms of a theory $V$. 
We define:
\begin{itemize}
\item
$U \rhd_{\sf A\text{-} loc}V$ iff there is a partial \compu\ $\Phi$, such that, whenever $V\vdash \phi$, we have
$\Phi(\phi)$ converges, say to $\tau$, and $U \vdash (\mathfrak{id}_V \wedge \phi)^\tau$.
\item
$U \rhd_{\sf B\text{-} loc}V$ iff there is a partial \compu\ $\Psi$, such that, whenever ${\sf W}_i = \verz \psi$ and $V\vdash \psi$, we have
$\Psi(i)$ converges, say to $\nu$, and $U \vdash (\mathfrak{id}_V \wedge \psi)^\nu$.
\item
$U \rhd_{\sf C\text{-} loc}V$ iff there is a partial \compu\ $\Theta$, such that, whenever ${\sf W}_j = \verz {\theta_0,\dots,\theta_{k-1}}$ and $V\vdash \theta_s$,
for all $s<k$, we have
$\Theta(j)$ converges, say to $\rho$, and $U \vdash \mathfrak{id}^\rho_V$ and  $U \vdash \theta_s^\rho$, for $s<k$.
\end{itemize}

We show that A-local and B-local coincide with local and that C-local coincides with global.

It is immediate that, if $U \rhd_{\sf B\text{-} loc}V$, then $U \rhd_{\sf A\text{-} loc}V$. It is equally immediate that
if $U \rhd_{\sf A\text{-} loc}V$, then $U \rhd_{\sf loc}V$. 
Suppose $U \rhd_{\sf loc}V$.
We show $U \rhd_{\sf B\text{-} loc}V$. Consider any index $i$. We enumerate ${\sf W}_i$. As soon as we find any $\psi$ in ${\sf W}_i$,
we run though the $U$-proofs to find a conclusion of the form $(\mathfrak{id}_V \wedge \psi)^\nu$. If we find such, we take
$\Psi(i):= \nu$. It is easy to see that $\Psi$ witnesses $U \rhd_{\sf B\text{-} loc}V$.

Clearly $U \rhd V$ implies $U \rhd_{\sf C\text{-} loc}V$. We prove the converse direction.
Suppose $\upsilon_0,\upsilon_1, \dots$ enumerates the theorems of $V$.
Let $\Theta$ be given as a witness of
$U \rhd_{\sf C\text{-} loc}V$. Using the Recursion Theorem we find $j^\ast$ such that
${\sf W}_{j^\ast}$ is given as follows. We first compute $\Theta(j^\ast)$. As long as it does not converge, we put nothing in
${\sf W}_{j^\ast}$. As soon as $\Theta(j^\ast)$ converges, say to $\rho^\ast$, we add $\upsilon_0$ to ${\sf W}_{i^\ast}$ and
search for a $U$-proof of $\upsilon_0^{\rho^\ast}$. As long as we don't find such, we add nothing more to ${\sf W}_{j^\ast}$.
As soon as we do find such a proof, we add $\upsilon_1$ to ${\sf W}_{j^\ast}$. Etcetera.
It is now easy to see that $\Theta(j^\ast)$ converges, ${\sf W}_{j^\ast} = V_{\mf p}$, and $\rho^\ast$ witnesses $U \rhd V$.


\section{Effective Essential Tolerance}\label{effesstol}

In~\cite{Vis22}, Albert Visser studies \emph{essential tolerance}, a reduction relation that backwards preserves
essential hereditary undecidability. The results of that paper have precise effective counterparts. The reduction
relation \emph{effective essential tolerance} backwards preserves effective essential hereditary creativity.

\begin{defn}
   \emph{$U$ effectively essentially tolerates $V$} or $U \jump_{\sf eff} V$ iff there are partial \compu\ functions
    $\Phi_0$ and $\Phi_1$, such that, whenever ${\sf W}_i$ is a consistent extension of $U$,
    $\Phi_0(i)$ and $\Phi_1(i)$ converge and ${\sf W}_{\Phi_0(i)}$ is a consistent extension of ${\sf W}_i$ and
    $\Phi_1(i)$ codes a translation $\tau$ that witnesses ${\sf W}_{\Phi_0(i)}\rhd V$.
\end{defn}

We can always assume that the witnesses of effective essential tolerance are total. Let $U$ be given with some index $j$.
We start with the partial witnesses $\Phi_k$ for $k = 0, 1$. Say, the total ones will be $\Psi_k$.
Consider any index $i$. We can, from $i$, effectively find an index $i^\ast$ of
$U \cup {\sf W}_i$. We now compute in parallel $\Phi_0(i^\ast)$, $\Phi_1(i^\ast)$ and we search, again in parallel,
for a contradiction in $U \cup {\sf W}_i$. If we find a value of a $\Phi_k(i^\ast)$ first, we set $\Psi_k(i)$ to that
value. If we find an inconsistency in $U \cup {\sf W}_i$, we set the $\Psi_k(i)$ that do not have a value yet to
some random value.

\newcommand{\sva}{\phi}
\newcommand{\svb}{\psi}
\newcommand{\svc}{\chi}
\newcommand{\idtb}[2]{\mathfrak{id}_{#1}^{#2}}

\medskip
We first verify some basic properties of $\jump_{\sf eff}$.

\begin{thm}\label{jumpysmurf}
    The relation $\jump_{\sf eff}$ extends $\rhd$.
\end{thm}

\begin{proof}
    Suppose $U\rhd V$ as witnessed by $\tau_0$. We take $\Phi_0(i) := i$ and $\Phi_1(i) := \tau_0$.
\end{proof}

\begin{thm}
The relation $\jump_{\sf eff}$ is reflexive and transitive. 
\end{thm}
\begin{proof}
Reflexivity is trivial. (It also follows from Theorem~\ref{jumpysmurf}.) We prove transitivity.
    Suppose $U \jump_{\sf eff} V \jump_{\sf eff} W$ and let $\Phi_0, \Phi_1$ witness $U \jump_{\sf eff} V$ and $\Psi_0, \Psi_1$ witness $V \jump_{\sf eff} W$. 
    Let ${\sf W}_i$ be any consistent extension of $U$. 
    Let $j : = \Phi_0(i)$ and $\tau: = \Phi_1(i)$, then ${\sf W}_j$ is a consistent extension of ${\sf W}_i$ and $\tau$ witnesses ${\sf W}_j \rhd V$. 
    Let $Y : = \{\phi \mid {\sf W}_j \vdash \phi^\tau\}$. 
    We can effectively find an index $p$ of $Y$ from $j$ and $\tau$. 
    Let $k : = \Psi_0(p)$ and $\mu : = \Psi_1(p)$. Then, ${\sf W}_k$ is a consistent extension of $Y$ and $\mu$ witnesses ${\sf W}_k \rhd W$. 
    Let $Z : = {\sf W}_j \cup \{\phi^\tau \mid {\sf W}_k \vdash \phi\}$. 
    It follows from the consistency of ${\sf W}_k$ that $Z$ is a consistent extension of ${\sf W}_i$. 
    Since $\tau$ witnesses $Z \rhd {\sf W}_k$, the composition $\tau \circ \mu$ witnesses $Z \rhd W$. 
    We can effectively find an index $q$ of $Z$ from $j$, $k$ and $\tau$. 
    We define $\Theta_0(i) : = q$ and $\Theta_1(i) : = \tau \circ \mu$. 
    Thus, $\Theta_0, \Theta_1$ witness $U \jump_{\sf eff} W$. 
\end{proof}

In~\cite{Vis22}, it is shown that the non-effective version $\jumpb$ is strictly between model-interpretability and local interpretability.
Since, of course, $V\jumpb_{\sf eff}U$ implies $V\jumpb U$ and $V\jumpb U$ implies $V\lhd_{\sf loc}U$, we
find that $V\jumpb_{\sf eff}U$ implies $V\lhd_{\sf loc}U$. 
As we have seen, in Appendix~\ref{locosmurf}, we can view local interpretability as its own effective version. So this result can be viewed as an
implication between effective notions.
Regrettably, we are not aware of a good effective version of model interpretability, so the implication from $V \lhd_{\sf mod} U$
to $V\jumpb U$ seems to have no good effective analogue.

We proceed to show the retro-transmission of salient properties.

\begin{thm}\label{reversosmurf}
Let $U$ be \ce~and consistent.
\begin{enumerate}[i.]
\item
Suppose $V$ is an effectively essentially incomplete \ce~theory and $U\jump_{\sf eff} V$.
 Then, $U$ is effectively essentially incomplete.
\item
Suppose $V$ is an effectively  essentially hereditarily creative \ce~theory and $U\jump_{\sf eff} V$.
 Then, $U$ is effectively essentially hereditarily creative.
 \end{enumerate}
\end{thm}

\begin{proof}
Ad (i).
Let $\Phi$  witness that $V$ is an effectively essentially incomplete and let $\Psi_0$, $\Psi_1$ witness that  $U\jump_{\sf eff} V$.
Let ${\sf W}_i$ be a consistent extension of $U$. Let $j := \Psi_0 i$ and $\tau := \Psi_1 i$. Then ${\sf  W}_j$ is a consistent extension of
${\sf W}_i$ and $\tau$ witnesses that ${\sf W}_j \rhd V$.  Let $\mc Z := \verz{\psi \mid {\sf W}_j \vdash \psi^\tau}$. 
We can clearly effectively find an index $k$ of $\mc Z$ from $j$ and $\tau$. Let $\chi := \Phi(k)$.
Then $\xi :=\chi^\tau$ is independent of ${\sf W}_j$ and, \emph{a fortiori}, of ${\sf W}_i$.
Inspecting the argument, we see that $\xi$ can be effectively found from $i$.

\medskip
Ad (ii).
Suppose $V$  is an effectively essentially hereditarily creative \ce~theory and $U\jump_{\sf eff} V$. Let $\Phi$ witness
the effective essential hereditary creativity of $V$ and let $\Psi_0$, $\Psi_1$ witness $U\jump_{\sf eff} V$.
We are looking for a witness $\Theta$ of the effective essential hereditary creativity of $U$.

Suppose ${\sf W}_i$ is a \ce~theory in the language of $U$ and  $U' := U \cup {\sf W}_i$ is consistent.
Suppose further that ${\sf W}_{i \mf p} \cap {\sf W}_k= \emptyset$. We need that $\Theta(i,k) \not\in {\sf W}_{i \mf p} \cup {\sf W}_k$.

Let $s$ be an index of $U'$
and let  $j := \Psi_0 s$ and let $\tau := \Psi_1 s$. So, ${\sf W}_j$ is consistent and extends $U'$. Moreover, 
 $\tau$ witnesses that ${\sf W}_j$ interprets $V$.
 \begin{itemize}
 \item
Let  $\mc Z:=  \verz{\svb \mid {\sf W}_i+\idtb V\tau \vdash \svb^\tau}$. Let $p$ be an index of $\mc Z$.
We have $\mc Z \vdash \svb$ iff  ${\sf W_i} + \idtb V\tau \vdash \svb^\tau$. Moreover, since ${\sf W}_j$ is a consistent extension of ${\sf W}_i + \idtb V\tau$, we find that $Z$ is consistent with $V$. 
\item
Let $\mc X := \verz{\phi \mid  (\idtb V\tau \to \phi^\tau)\in {\sf W}_k}$. Let $q$ be an index of $\mc X$.
Suppose $\phi \in \mc Z_{\mf p} \cap \mc X$. Then, $(\idtb V\tau \to \phi^\tau)\in {\sf W}_{i \mf p}$ and
$(\idtb V\tau \to \phi^\tau)\in {\sf W}_k$. \emph{Quod non.} 
\end{itemize}

We may conclude that $\chi := \Phi(p,q) \not \in \mc Z_{\mf p} \cup \mc X$. We find
$(\idtb V\tau \to \chi^\tau) \not\in {\sf W}_{i\mf p} \cup {\sf W}_k$. We found $p$ and $q$ effectively from
$i$ and $j$. So, we can set $\Theta(i,j) := (\idtb V\tau \to \Phi(p,q)^\tau)$.
\end{proof}

In~\cite{Vis22}, the notion of $\Sigma^0_1$-friendliness was developed. Inspecting the proof of \cite[Theorem 35]{Vis22}
and what is said directly below the proof, we see that if $U$ is $\Sigma^0_1$-friendly, then $U \jump_{\sf eff} {\sf R}$.
This provides a nice source of examples of theories that are effectively essentially hereditarily creative. Specifically, the
theory ${\sf PA}^-_{\sf scatt}$ studied in~\cite{Vis22} is effectively essentially hereditarily creative.

In Appendix A of~\cite{Vis22} it is shown that we can extend essential tolerance by considering theory extensions that
allow addition of finitely many constants. This notion yields earlier results by Vaught in his~\cite{Vau62}.
We did not pursue this avenue yet, but, \emph{prima facie}, there seem to be no obstacles to extend the results of this
appendix to this wider notion. 

\end{document}